\documentclass[a4wide,12pt]{article}

\usepackage{amsmath}
\usepackage{a4wide}
\usepackage{amsfonts}
\usepackage{amsthm}
\usepackage{tikz-cd}
\usepackage{booktabs}
\usepackage{enumerate}
\usepackage{subfigure}

\newcommand{\Pb}{{\mathbb P}}
\newcommand{\R}{{\mathbb R}}
\newcommand{\E}{{\mathbb E}}
\newcommand{\var}{{\mathbb{V}\mathrm{ar}}}
\newcommand{\cov}{{\mathbb{C}\mathrm{ov}}}

\newcommand{\tZ}{{\widetilde Z}}

\newtheorem{lemma}{Lemma}
\newtheorem{theorem}{Theorem}
\newtheorem{proposition}{Proposition}

\newtheorem{problem}{Problem}
\theoremstyle{definition}
\newtheorem{definition}{Definition}

\newtheorem{remark}{Remark}
\newtheorem{corollary}{Corollary}

\DeclareMathOperator{\ind}{1{\hskip -2.5 pt}\mathrm{I}}
\newcommand{\beq}{\begin{equation}}
\newcommand{\eeq}{\end{equation}}
\newcommand{\beqn}{\begin{eqnarray}}
\newcommand{\eeqn}{\end{eqnarray}}
\newcommand{\rosso}[1]{{\color{black}{ #1}}}

\def\mb{\textcolor{black} }

\begin{document}

\title{{\bf \Large SDEs with uniform distributions: peacocks, \\ conic martingales and ergodic uniform diffusions}}
\author{ Damiano Brigo\thanks{Department of Mathematics, Mathematical Finance and Stochastic Analysis groups, Imperial College London, {\small \tt{damiano.brigo@imperial.ac.uk}}
}
\and Monique Jeanblanc\thanks{Laboratoire de Math\'ematiques et Mod\'elisation d'\'Evry (LaMME),   Universit\'e d'\'Evry-Val-d'Essonne, UMR CNRS 8071
{\small \tt{monique.jeanblanc@univ-evry.fr}}}
\and  Fr\'ed\'eric Vrins\thanks{
 Louvain Finance Center (LFIN) and Center for Operational Research and Econometrics (CORE),  UCLouvain, {\small \tt{frederic.vrins@uclouvain.be}}}
}

\date{First version: April 4, 2016. This version: \today}

\maketitle

\begin{abstract}
It is known since Kellerer (1972) that for any process that is increasing for the convex order, or ``peacock" as in Hirsch et al. 2011 \cite{peacock}, there exist martingales with the same \mb{marginals laws}. Nevertheless, there is no general constructive method for finding such martingales that yields diffusions. We consider the \textit{uniform peacock}, namely the peacock with uniform law at all times on a generic \textit{time-varying support} $[a(t),b(t)]$. We derive explicitly the corresponding SDEs and prove that, under certain \rosso{``conic''} conditions on $a(t)$ and  $b(t)$, they admit a unique strong diffusive solution. \rosso{To guess the candidate SDE we resort to the approach of inverting the Fokker Planck equation. Dupire (1994) \cite{dupire} did this for volatility modeling. Here we tackle the inversion with the caveats needed when dealing with uniform margins with conic boundaries. This was done originally in the unpublished preprint by Brigo (1999) \cite{BrigoIMI}.} Independently,  Madan and Yor (2002) \cite{Mada02}  obtained the result as a simple application of Dupire. Once the SDE is guessed, we analyze it rigorously, discussing cases \mb{ where our approach adds strong uniqueness  of the solution of the SDE and cases where only a weak solution is obtained.} We further study the local time and activity of the solution.
We then study the peacock with uniform law at all times on a \textit{constant support} $[-1,1]$ and derive the SDE of an associated mean-reverting diffusion process with uniform margins that is not a martingale. For the related SDE we prove existence of a solution. We derive the exact transition densities \rosso{for both the mean reverting and the original conic martingale cases}. We prove limit-laws and ergodic results:  the SDE solution transition law tends to be uniform after a long time. Finally, we provide a numerical study confirming the desired uniform behaviour. These results may be used to model random probabilities, recovery rates or correlations.
\end{abstract}

\medskip


\medskip

AMS classification codes: 60H10,  60J60

\pagestyle{myheadings}
\markboth{}{D. Brigo, M. Jeanblanc \& F. Vrins. SDEs with uniformly--distributed solutions.}		


\section{Introduction}
A \textit{peacock} is an integrable process that is increasing in the convex order. Equivalently, a peacock is a process with  (i) constant expected value and (ii) whose transform via any positive and convex function $\Psi$ has an increasing expectation (see Definition 1.3 in \cite{peacock}). More precisely, \mb{a peacock is a process $X$ with constant expected value such that $t \mapsto \E(\Psi (X_t))$  is } increasing for any  convex function $\Psi$ such that $\E(\vert \Psi (X_t)\vert)<\infty$ for all $t$.
 From this equivalent representation, it is trivial to show via the law of iterated expectations and Jensen's inequality that any martingale is a peacock.
%
Reciprocally, it is known from Kellerer~\cite{Kell72} that for any peacock there exist martingales (called \textit{associated martingales}) with the same marginal laws. Nevertheless, there is no guarantee that these associated martingales are diffusions. Moreover, specifying explicitly \mb{a class of   martingales} associated to a given peacock is not trivial. \medskip

In this paper, we provide the explicit dynamics of diffusion processes associated to the \textit{uniform peacocks} that is, the peacocks whose marginals have a uniform distribution on a time-varying support imposing, without loss of generality, $X_0=0$. \rosso{To that end, we study a family of regular diffusion martingales obtained from \cite{BrigoIMI} or \cite{Mada02}, with an approach reminiscent of \cite{dupire}. These martingales evolve on the expanding (``conic'') support $t \mapsto [-b(t),b(t)]$, $b(0)=0$. These diffusion martingales will be obtained via the Stochastic Differential Equation (SDE)}
\begin{equation*}
d X_t = \left(\ind_{\{X_t \in [-b(t),b(t)]\}} \frac{\dot{b}(t)}{b(t)} (b(t)^2 - X_t^2)\right)^{1/2} dW_t, \ \ X_0=0,
\end{equation*}
We show that, under adequate conditions on the boundary, this SDE admits a unique strong solution which is associated to the uniform peacock. This extends previously known results where $b(t)$ is for example equal to $t$ (p. 252 in~\cite{peacock}), adding strong uniqueness. Our result allows one to  show strong existence and uniqueness for the case $b(t) = t^\alpha$, $\alpha \ge 1$. The case  $b(t)=\sqrt{t}$ has to be dealt with using different techniques. For cases like $b(t)=\sqrt{t}$ we use the approach in p. 253--260 of~\cite{peacock}, and in that case we can only obtain uniqueness in law. We further show that the solution processes spend zero time at the boundaries.

\medskip

The above diffusion coefficient was initially guessed by informally inverting the forward Kolmogorov (also known as Fokker-Planck) equation, when forcing the marginal density of the solution $X$ to be uniform at all times with support $[-b,b]$ as initially sketched in the preprint \rosso{by Brigo (1999) \cite{BrigoIMI}. This inversion is presented in Section \ref{sec:informal}. }
%
%
\rosso{This inversion technique was used in the past by the first named author to construct diffusion processes with densities in exponential families \cite{brigogyor,brigo00} and has been used more generally in a variety of contexts in mathematical finance, especially in volatility modelling. For example, the earlier Dupire (1994) \cite{dupire} finds the diffusion coefficient (``local volatility") that is consistent with a probability law extrapolated from a surface of option prices. The paper \cite{Brigobachelier} deals with designing a diffusion process consistent with a mixture of distributions for volatility smile modeling, whereas \cite{brigomercurio} inverts the Kolmogorov equation to show how two stochastic processes with indistinguishable laws in a time grid under the historical measure can lead to arbitrarily different option prices, possibly explaining the differences between historical and implied volatility.}

\rosso{However, in the context of peacocks and uniforms, the first published reference proposing the inversion of the Fokker Planck equation to obtain the uniform peacock is Madan and Yor (2002) \cite{Mada02} where the authors, partly building on Dupire (1994) \cite{dupire}, present three different construction schemes to find martingales associated to a given peacock, with full proofs. These three methods are: \textit{Skorohod embedding}, \textit{inhomogeneous independent increments} and \textit{continuous martingales}. The first two methods provide martingales taking the form of time-changed Brownian motions. The last method consists of inverting the Fokker-Plank equation (also known as forward Kolmogorov equation) and leads to the same solution as in \cite{BrigoIMI}, as it can be used to derive immediately the result in Section \ref{sec:informal}. The methods in \cite{Mada02}}
 \mb{ give the form of the volatility coefficient, and the authors point out that the diffusion exists in the standard case where  the diffusion coefficient is Lipschitz. Most }\rosso{of the cases we analyze will not satisfy this assumption}.


The paper is structured as follows. In Section \ref{sec:conicmarting} we formulate the problem. A solution is attempted in Section \ref{sec:informal} along the lines of the above mentioned inversion. We then study the solution rigorously in Section \ref{sec:rigorous} and prove that the related SDE admits a unique strong solution. We further prove that the solution has indeed a uniform distribution with the desired conic boundary. Being bounded on a finite horizon, the solution is thus a genuine martingale associated to the uniform peacock. 
In Section \ref{sec:meanrev} we re-scale the conic diffusion martingale and study the related mean-reverting uniform diffusions, where now the uniform law is not conic but constant. Two special cases of interest are standard uniforms and uniforms in $[-1,1]$, which can be used to model for example maximum--entropy recovery rates or random probabilities and random correlations, respectively.

In Section \ref{sec:ergodicth} we prove limit-law results for the mean-reverting uniform $[-1,1]$ re-scaled process. In doing so, we derive the exact transition density of the SDE solution. While we know that the solution margins are uniform by constructions, this will not hold for the transition densities in general and we characterize them via their moments. We also prove a limit-law result showing that after a long enough time any initial condition at a given time in the transition density is forgotten and the limit tends again to a uniform. We show that a particular case of the boundary $b(t)$ leads to an ergodic diffusion process, and that under reasonable regularity all other cases are deterministic time changes of this ergodic diffusion.

In Section \ref{sec:localtime} we further show that the rescaled processes have zero local time at the boundaries $-1$ and $1$.
In Section \ref{sec:specific} we revisit the two previously known cases and hint at new choices for the boundaries.
In the linear case we study the process pathwise activity, finding that the pathwise activity of the mean reverting diffusion vanishes asymptotically. The behavior of the process is illustrated based on numerical simulations that confirm our earlier characterization of the SDE having the desired marginal distribution and our limit-law type results.



\section{Conic diffusion martingales with uniform distribution}\label{sec:conicmarting}
We set out to construct a martingale diffusion process $X $ (zero drift), i.e. a diffusion process driven by a Brownian motion that is a martingale, with marginal at time \mb{$t>0$} having a uniform distribution in an interval $[a(t),b(t)]$. The martingale condition implies that $\E[X_t]  = \E[X_0]$ for all $t \ge 0$, whereas the uniform distribution requirement implies that $\E[X_t] = [a(t)+b(t)]/2$ for all $t \ge 0$. Thus we have
$a(t) + b(t) = a(0) + b(0)$ for all $t \geq 0$. We will assume $a(0)=b(0)=0$, taking the initial condition $X_0$ to be deterministic and with value zero (Dirac delta law in $0$).  Hence $b(t) = -a(t)$ for all $t\ge 0$.

With such preliminaries in mind, we state the following

\begin{problem}[Designing conic martingale diffusions with given uniform law]\label{pr:conic}
Consider the diffusion process
\begin{equation}\label{eq:martinX} d X_t = \sigma(X_t,t) dW_t, \ \ X_0 =0.
\end{equation}
Find a diffusion coefficient $\sigma(x,t)$ such that
\begin{enumerate}
\item The SDE \eqref{eq:martinX} has a unique strong solution;
\item The solution of \eqref{eq:martinX} at time $t>0$ is uniformly distributed in $[-b(t),b(t)]$ for a non-negative strictly increasing continuous function $t \mapsto b(t)$ with $b(0)=0$.
\end{enumerate}
\end{problem}
In other terms, our aim is to build a diffusion martingale $X $ as in \eqref{eq:martinX}
such that the process $X $ has a density $p(x,t)$ at time $t>0$ at the point $x$ given by the uniform density
\begin{equation}\label{eq:uniformp}
\rho(x,t) := \ind_{\{x \in [-b(t),b(t)]\}} / (2 \ b(t)) .
\end{equation}
We call such martingales ``conic" because their support opens up in time.

In Problem \ref{pr:conic}, $b$ is restricted to be strictly increasing in time. The reason is that the tight upper (resp. lower) bound of any bounded martingale must be a non-decreasing (resp. non-increasing) function (\cite{vrins}). Hence, $X $ is a \textit{conic martingale}; it is a martingale that exhibits a conic behavior. We will need strict monotonicity in the following derivation, so we assumed $b$ to be strictly increasing in Problem \ref{pr:conic}.

\section{Deriving the candidate SDE for a uniformly distributed martingale}\label{sec:informal}

\rosso{We present the approach in the preprint \cite{BrigoIMI}, although the same guess could be derived by applying results in the published paper \cite{Mada02}}. Let us now guess a candidate solution $\sigma$ for Problem~\ref{pr:conic}. To do this, we write the forward Kolmogorov (or Fokker Planck) equation for the density $p$ of \eqref{eq:martinX}, impose $\rho$ to be a solution and derive the resulting $\sigma$. The derivation is informal but it is given full mathematical rigor by showing later that the resulting SDE \eqref{eq:martinX} has a unique strong solution and confirming further, via moments analysis, that the density is indeed uniform.

The forward Kolmogorov eq. for \eqref{eq:martinX} with $\rho$ plugged in as a solution reads
\begin{equation}\label{eq:fokkerrho1}
\frac{\partial \rho(x,t)}{\partial t}= \frac{1}{2} \frac{\partial^2}{\partial x^2} (\sigma(x,t)^2 \rho(x,t)), \ \ \rho(x,0) = \delta_{0}(x).
\end{equation}
Now we integrate twice both sides of \eqref{eq:fokkerrho1} with respect to $x$ and assume we can switch integration with respect to  $x$ and differentiation with respect to  $t$ (one can check \textit{a posteriori} that the solution we find has a continuous partial derivative with respect to $t$ so that Leibniz's rule can be used). We obtain
\begin{equation}\label{eq:fokkerrho2}
\frac{\partial}{\partial t}\left( \int_{-\infty}^x \left(\int_{- \infty}^y \rho(z,t) dz\right)dy\right)= \frac{1}{2}  \sigma(x,t)^2 \rho(x,t),
\end{equation}
assuming the relevant first and second derivatives with respect to  $x$ on the right hand side vanish fast enough at minus infinity. Compute for $t> 0$, substituting from \eqref{eq:uniformp},
\[\varphi(x,t):= \int_{-\infty}^x \left(\int_{- \infty}^y \rho(z,t) dz\right)dy
=
\begin{cases}
    0,& \text{if } x < -b(t)\\
    \frac{(x+b(t))^2}{4 b(t )}, &  \text{if } x \in [ -b(t), b(t)]\\
    x  ,        & \text{if } x  > b(t)
\end{cases}
\]
and note that $\varphi$ is continuous in $x$. Equivalently,
\begin{equation}\label{eq:doubleprimitive} \varphi(x,t) = \frac{(x+b(t))^2}{4 b(t)} \ind_{\{x \in [-b(t),b(t)]\}} +
x \ind_{\{ x>b(t)\}}.
\end{equation}
Thus, rewriting \eqref{eq:fokkerrho2} as
\begin{equation}\label{eq:fokkerrho3}
\frac{\partial \varphi(x,t)}{\partial t}= \frac{1}{2}  \sigma(x,t)^2 \rho(x,t),
\end{equation}
and substituting \eqref{eq:doubleprimitive} we are done.
To do this, we need to differentiate $\varphi$ with respect to time.
The calculations are all standard but one has to pay attention when differentiating terms in \eqref{eq:doubleprimitive} such as
\[ \ind_{\{x \in [-b(t),b(t)]\}} = \ind_{\{x\ge-b(t)\}} - \ind_{\{x > b(t)\}} \]
which can be differentiated in the sense of distributions, 
\[ \frac{d}{dt} \ind_{\{x>b(t)\}} = \frac{d}{dt} \ind_{\{t < b^{-1}(x)\}} = - \delta_{b^{-1}(x)}(t) \]
where the index in $\delta$ denotes the point where the Dirac delta distribution is centered.
One can check that all terms involving $\delta$'s either offset each other or are multiplied by a function that vanishes at the point of evaluation.

Assuming $b$ is differentiable, omitting time arguments and denoting differentiation with respect to  time with a dot one gets:
\[ \frac{\partial \varphi(x,t)}{\partial t} = - \frac{-\dot{b} (2 b) (x+b) + (2 \dot{b}) (x+b)^2}{2(2 b)^2} \ind_{\{x \in [-b,b]\}} . \]
We notice that $\dot{b}$ appears only in ratios $\dot{b}/b$, so that this quantity may be extended to time $t=0$ by continuity if needed provided that the limit exists.

The above quantity is the left hand side of \eqref{eq:fokkerrho3}. We can substitute $\rho$ on the right hand side and we have that
\[- \frac{-\dot{b} (2b) (x+b) + (2\dot{b}) (x+b)^2}{2(2b)^2} \ind_{\{x \in [-b,b]\}}  = \frac{1}{2} \frac{\sigma(x,t)^2}{2b} \ind_{\{x \in [-b,b]\}}.\]
After some algebra, one obtains
\[ \sigma^2(x,t) = \ind_{\{x \in [-b(t),b(t)]\}} \frac{\dot{b}(t)}{b(t)} (b(t)^2 - x^2).\]

From the above development, we expect the diffusion coefficient $\sigma(x,t)$ defined as
\begin{equation}
\sigma(x,t) := \ind_{\{x \in [-b(t),b(t)]\}} \sqrt{\frac{\dot{b}(t)}{b(t)} (b(t)^2 - x^2)} \label{eq:dc}
\end{equation}
to be a valid candidate for the solution $X $ of \eqref{eq:martinX} to be a martingale with marginals having a uniform law in $[-b,b]$. In order to rigorously show that, we prove in the  next section that, under suitable regularity condition on the boundaries $t \mapsto b(t)$, the SDE \eqref{eq:martinX} with diffusion coefficient (\ref{eq:dc}) admits a unique strong solution and that this solution has indeed a uniform law at all times. In the more  general case where regularity of the boundary is relaxed we prove that the solution is unique in law.


\section{Analysis of the SDE: solutions and distributions}\label{sec:rigorous}

%
\begin{theorem}[Existence and Uniqueness of Solution for candidate SDE solving Problem \ref{pr:conic}]\label{th:exist}
Let $T>0$ and $b$ be a strictly increasing function defined on $[0,T]$, continuous in $[0,T]$ and continuously differentiable in $(0,T]$ and satisfying $b(0)=0$. Assume  $\dot{b}$ to be bounded in $(0,T]$.  The stochastic differential equation
\begin{equation}
\label{eq:martinX2} d X_t = \ind_{\{X_t \in [-b(t),b(t)]\}} \left( \frac{\dot{b}(t)}{b(t)} (b(t)^2 - X_t^2)\right)^{1/2} dW_t, \ \ X_0=0,
\end{equation}
whose diffusion coefficient is extended to $t=0$ by continuity via
\[ \sigma(x,0):=0 \ \ \mbox{for all} \ \ x ,\]  admits a unique strong solution and its solution $X $ is distributed at every point in time $t$ as a uniform distribution concentrated in $[-b(t),b(t)]$. We thus have a conic diffusion martingale with the cone expansion controlled by the time function $b$.
Moreover,
one can show that the solution processes spend zero time at the boundaries $-b$ and $b$. \end{theorem}

\begin{proof}

By continuity of diffusion paths, the solution $X$ to the SDE~(\ref{eq:martinX2}), if it exists, belongs to $[-b(t),b(t)]$ almost surely, \mb{since the square root must be well defined}. Indeed, the diffusion coefficient $\sigma(t,x)$ vanishes at the boundaries $\{-b(t),b(t)\}$. Because $b(t)$ is increasing, the process cannot exit the cone $[-b(t),b(t)]$.

It remains to prove that the solution $X$ to~(\ref{eq:martinX2}) exists and is unique. To that end, it is enough to show that
$\sigma(x,t)$ satisfies the linear growth bound and is Holder-$1/2$ for all $t\in[0,T]$~\cite{Kara05}.

Clearly, $\sigma(x,t)$ in (\ref{eq:dc}) satisfies the linear growth bound since it is uniformly bounded on $[0,T]$. To see this, notice that
\[ 0 \le \sigma^2(x,t)=\ind_{\{ -b(t)\leq x\leq b(t)\}}(\dot{b}(t)/b(t))(b^2(t)-x^2)\leq \dot{b}(t)b(t) \ \ \mbox{for all} \ \ x, \]
and that $\dot{b}(t)b(t)$ is bounded on $(0,T]$ by assumption, with zero limit when $t \downarrow 0$.  This allows us to conclude that
\[ \lim_{t \downarrow 0} \sigma^2(x,t) = 0 \ \ \mbox{for all} \ \ x.\]
Since $\sigma(x,t)^2$ is continuous and bounded on $(0,T]$ with the above limit, it admits a continuous extension at $t=0$ taking value zero. The extended $\sigma(x,t)$  is unique and  uniformly bounded on $[0,T]$.

We now proceed with the Holder continuity of $\sigma$. 
Of course, $f(x)=\sqrt{\vert x\vert}$ is Holder-$1/2$ on $\mathbb{R}$ since $\vert\sqrt{\vert x\vert}-\sqrt{\vert y\vert}\vert\leq \sqrt{|x-y|}$ for all $x,y$. We now check  that $\sigma(t,x)$  is  Holder-$1/2$ uniformly in $t>0$ ($t=0$ is not a problem given the above extension by continuity). See also~\cite{vrinsjeanblanc}).

Define $I(t):= [-b(t),b(t)]$. We check the possible cases.
\begin{enumerate}
\item If  $x, y \notin I(t)$, the diffusion coefficient vanishes and one gets $\vert \sigma (t,x)-\sigma (t,y)\vert =0$
\item If $x, y \in I(t)$, using the Holder-$1/2$ continuity of $\sqrt{|x|}$ :
\beqn
\vert \sigma (t,x)-\sigma (t,y)\vert &=&\sqrt{\frac{\dot{b}(t)}{b(t)}} \vert \sqrt{b^2(t)-x^2}-\sqrt{b^2(t)-y^2}\vert\nonumber\\
&\leq& \sqrt{\frac{\dot{b}(t)}{b(t)}}\sqrt{\vert (b^2(t)-x^2)-(b^2(t)-y^2)\vert}\nonumber\\
&=&\sqrt{\frac{\dot{b}(t)}{b(t)}}\sqrt{\vert y^2-x^2\vert}\leq \sqrt{\frac{\dot{b}(t)}{b(t)}}\sqrt{\vert y+x\vert}\sqrt{\vert y-x\vert}\leq \sqrt{\frac{\dot{b}(t)}{b(t)}} \sqrt{2b(t)}\sqrt{\vert x-y\vert}\nonumber\\
&=& \sqrt{2\dot{b}(t)}\sqrt{\vert x-y\vert}
\eeqn
and we are done since $\dot{b}$ is assumed to be bounded in $(0,T]$.

\item If $x\in I(t),\; y>b(t) $: $$\vert \sigma (t,x)-\sigma (t,y)\vert=\vert \sigma (t,x)\vert=\sqrt{\frac{\dot{b}(t)}{b(t)}}\sqrt{b^2(t)-x^2}= \sqrt{\frac{\dot{b}(t)}{b(t)}}\sqrt{b(t)+x}\sqrt{b(t)-x}\leq$$
$$\le \sqrt{\frac{\dot{b}(t)}{b(t)}}\sqrt{2 b(t)} \sqrt{b(t)-x} \le \sqrt{2 \dot{b}(t)} \sqrt{ \vert x-y\vert} $$
and again we are done since $\dot{b}$ is bounded in $t$.
\item If $x\in I(t),\; y<-b(t) $ (so that $-y>b(t)$) :
\beqn
\vert \sigma (t,x)-\sigma (t,y)\vert &=&\vert \sigma (t,x)\vert\leq \sqrt{\frac{\dot{b}(t)}{b(t)}} \sqrt{b(t)-x}\sqrt{b(t)+x} \leq  \nonumber \\ &\leq& \sqrt{\frac{\dot{b}(t)}{b(t)}} \sqrt{2b(t)}\sqrt{x+b(t)}\leq  \sqrt{2 \dot{b}(t)}\sqrt{x-y}\nonumber
\eeqn

\item The case $x\notin I(t),\; y  \in I(t) $  is similar to steps 3 and 4.
\end{enumerate}

Hence, the solution $X$ to~(\ref{eq:martinX2}) exists and is unique. Because it is bounded and evolves between $-b(t)$ and $b(t)$, it is a conic $[-b(t),b(t)]$-martingale.

Finally, the fact that solutions spend zero time at the boundaries $-b$ and $b$ will be proven in Theorem \ref{th:attain} below.

\end{proof}

\begin{remark}[Indicator function in the diffusion coefficient] We notice that the diffusion coefficient vanishes for $x=\pm b(t)$, that diffusion paths are continuous and that the boundary is expanding. It follows that even if we omit the indicator in the diffusion coefficient expression, the related SDE will not leave the cone $[-b,b]$. Therefore, one could omit the indicator whenever the diffusion coefficient is featured inside a SDE.
\end{remark}

We have proven that the SDE \eqref{eq:martinX2} has a unique strong solution. The SDE itself has been obtained by inverting the Kolmogorov equation for a uniform marginal density at time $t$ in $[-b(t),b(t)]$, so we expect the density of the solution to be that uniform distribution. However, we haven't proven that the forward Kolmogorov equation for the density of \eqref{eq:martinX2} has a unique solution.  To prove that our SDE \eqref{eq:martinX2} has the desired uniform distribution, one resorts to a characterization of  the uniform distribution by its moments, showing that the moments of the solution of \eqref{eq:martinX2} are the same as the moments of the desired uniform law, and showing that this characterizes the uniform law. The latter is clearly related to Carleman's theorem, as it is well known that having uniformly bounded moments, the continuous uniform distribution on an interval $[a,b]$ with finite $a,b\in \R$ is determined by its moments, see for example Chapter 30 of \cite{billingsley95}. This proof is straightforward but we include it in Appendix \ref{app:moments} for completeness. A different approach is using Theorem \ref{th:existnouniq} below, since that is enough to guarantee a uniform distribution.

The special case $b(t) = k t$ gives us a conic martingale with uniform distribution where the boundaries grow symmetrically and linearly in time.  This example was considered originally in \cite{BrigoIMI} and is also in \cite{peacock} (see for instance ex. 6.5 p.253 with $\varphi(x)=x$ and \mb{$f(z)=1/2\ind_{\{-1\leq z\leq 1\}}$}). More generally, our result allows to treat the case $b(t) = t^\alpha$, for $\alpha \ge 1$. Staying in the class of boundaries $t^\alpha$, we see that the case $\alpha < 1$ violates our assumptions, since in that case $\dot{b}$ is not bounded in $0$, and has to be dealt with differently. For $1/2 \le \alpha<1$, and with the square root case in mind in particular, we now introduce a different approach to prove existence (but not uniqueness) of the SDE solution, as done in the peacock processes literature \cite{peacock}.

\begin{theorem}[Existence of Solution for SDE solving Problem \ref{pr:conic} under milder conditions on the boundary]\label{th:existnouniq}
Let $b$ a continuous strictly increasing function defined on $[0,T]$ and of class $C^1$ in $(0,T]$, with $b(0)=0$ and $T$ a positive real number. Assume  $b \dot{b}$ to be bounded in $(0,T]$.
The stochastic differential equation \eqref{eq:martinX2}, namely
\begin{equation*}
d X_t = \ind_{\{X_t \in [-b(t),b(t)]\}}\left( \frac{\dot{b}(t)}{b(t)} (b^2(t) - X_t^2)\right)^{1/2} dW_t,\ \  t>0, \ \  X_0=0,
\end{equation*}
 admits a weak solution that is unique in law and its solution $X $ is distributed at every point in time $t$ as a uniform distribution concentrated in $[-b(t),b(t)]$. We thus have a conic diffusion martingale with the cone expansion controlled by the time function $b$. If moreover $\dot{b} b$ admits a finite limit for $t \downarrow 0$ one can show that the solution processes spend zero time at the boundaries $-b$ and $b$.
\end{theorem}
\begin{proof}

By continuity of diffusion paths, the solution $X$ to the SDE~(\ref{eq:martinX2}), if it exists, belongs to $[-b(t),b(t)]$ almost surely.  The solution  of  \eqref{eq:martinX2}  has to be understood in a first step as a process satisfying, for any $t\geq \epsilon >0$
  $$X_t = X_\epsilon+\int_\epsilon^t  \ind_{\{x \in [-b(s),b(s)]\}}\left( \frac{\dot{b}(s)}{b(s)} (b^2(s) - X_s^2)\right)^{1/2} dW_s$$
where $X_\epsilon$ has a uniform law in $[-b(\epsilon),b(\epsilon)]$. The value of $X$ at time 0 is defined by continuity when $\epsilon$ goes to zero (we will prove that the limit exists), and  \eqref{eq:martinX2}  can be written   $X_t = \int_0^t  \ind_{\{X_s \in [-b(s),b(s)]\}}\left( \frac{\dot{b}(s)}{b(s)} (b^2(s) - X_s^2)\right)^{1/2} dW_s$ which has a meaning even if $\sigma (0,x)$ is not well defined.  The diffusion coefficient $\sigma(t,x)$ vanishes at the boundaries $\{-b(t),b(t)\}$ and because $b$ is increasing, it follows that $X_t\in[-b(t),b(t)]$ for all $t\geq 0$.  

It remains to prove that a solution $X$ to~\eqref{eq:martinX2} exists. We follow the methodology introduced in \cite{peacock}, see in particular Lemma 6.8 for the case where $h$ is the density of a uniform law on [-1,+1], and $a_h$ is defined in (6.49).  In this work the authors introduce a process $Y=(Y_t)_{t\in \R}$  such that, for all $t\geq s$
$$Y_t=Y_s -\frac{1}{2} \int_s^t   Y_u du+\frac{1}{\sqrt 2}\int_s^t  \sqrt{1-Y^2_u} dB_u$$
with marginals having uniform distribution  on [-1,+1], where $B$ is a Brownian motion on $\R$ (not merely $\R^+$), meaning that it is a process with continuous paths and stationary independent increments .
%
 Then, setting
 \begin{equation}\label{eq:XYgamma} X_t=  b(t)Y_{\gamma(t)}
 \end{equation}  for $t>0$, where $\gamma$ is an increasing differentiable  function, leads to a process with uniform marginals on $[-b(t),b(t)]$ (since by construction $Y_{\gamma(t)}$ has a uniform law). It remains to find $\gamma$ making $X$ a martingale with the prescribed dynamics.
Using \cite[lemma 5.1.3.]{3m}, and defining   $\beta( y):=  \frac{1}{\sqrt 2}\sqrt{1-y^2}$ and $U$ as $U_t:=\int_s
^{\gamma(t)}  \beta (Y_u)dB_u$, there exists a $\mathbb{F}=({\cal{F}}_{\gamma(t)})_{t\geq 0}$ Brownian motion $W$ such that
$$dU_t= \beta ( Y_{\gamma(t)})\, \sqrt{\dot{\gamma} (t)} dW_t\;.$$ It follows that
\begin{equation}\label{ygamma}  d_t Y_{\gamma(t)} = -\frac{1}{2} Y_{\gamma (t)}\dot{\gamma}(t) dt + \beta (Y_{\gamma (t)})\,\sqrt{\dot {\gamma }(t)}
 dW_t \end{equation} and by integration by parts
\beq
dX_t=
b(t) \beta ( Y_{\gamma(t)})\, \sqrt{\dot{\gamma}(t)}dW_t\label{eq:martinX3}
\eeq and the process $X$ is a local martingale. Equating the diffusion coefficient of \eqref{eq:martinX2} to that of \eqref{eq:martinX3} yields to identifying $\dot{\gamma} (t)= 2\frac{\dot{b}(t)}{b(t)}$ so that a valid choice for our time-change process is $\gamma (t)=2 \ln b(t)$.
The process $X$ is a true martingale: indeed by assumption on the boundedness of $b \dot {b}$
$$\sigma^2(x,t)=\ind_{\{ -b(t)\leq x\leq b(t)\}}(\dot{b}(t)/b(t))(b^2(t)-x^2)\leq \dot{b}(t)b(t)\leq C $$
and hence  $$\E\left[ \left(\int_ s ^t  \sigma (u, X_u)dW_u\right)^2\right]= \E\left[ \int_ s ^t  \sigma^2 (u, X_u)du \right] \leq  C (t-s).$$
It remains to prove that $X_t=b(t) Y_{2\ln b(t)}$ goes to 0 a.s. when $t$ goes to 0  which is similar to the proof given in \cite{peacock}. Again in \cite{peacock} it is shown that one has uniqueness in law and the argument can be straightforwardly repeated for our process here.
Finally, the claim on the time spent at the boundaries is proven in Theorem \ref{th:attain}
.
\end{proof}

%

\section{Mean reverting uniform diffusions with constant boundaries}\label{sec:meanrev}
In this paper we define mean reversion as follows.
A real-valued squared-integrable  Markov process $(\xi_t(\omega))_t$ mean reverts towards a long term mean $\bar{\theta} \in \mathbb{R}$ if the following holds: for all $s$ in the time domain of the process and all possible values $\bar{\xi}$ for the process at time $s$, one has
\[ \lim_{t \uparrow \infty} \mathbb{E}[\xi_t | \xi_s = \bar{\xi}] = \bar{\theta} \]
where $\bar{\theta}$ is a deterministic constant.
This condition implies that wherever the process state is found at a given future time, the long term mean from that time onward is a constant deterministic value that does not depend on the chosen time and state. We also require
$\lim_{t \uparrow \infty} \mathbb{V}ar(\xi_t)$ to exist finite.

Mean reversion is an important property that tells us that the process expectation  tends to forget a specific initial condition in the long run from any past time. However, it is a special case of a more general property. If we assume that the process has a density with respect to the Lebesgue measure at all times $t>0$, denote by
\[  p_{\xi_t|\xi_s}(x;y)\ dx = \mathbb{P}\{ \xi_t \in dx | \xi_s = y\} \]
the conditional density of $\xi_t$ at $dx$ given $\xi_s=y$, with $s<t$. We have that the whole law forgets earlier conditions if
\[ \lim_{t \uparrow +\infty} p_{\xi_t|\xi_s}(x;y) \ \mbox{exists, is a density in $x$ and depends neither on} \ s \ \mbox{nor on} \ y.\]
We now focus on mean reversion and will get to the general law later in Section \ref{sec:ergodicth}.

Take ${t_0} >0$ and consider the solution of the SDE \eqref{eq:martinX2} for $t \ge {t_0}$.  If one starts from $X$, solution of  \eqref{eq:martinX2}, one immediate way to obtain a diffusion with a standard uniform distribution at all times is to re-scale $X_t$ by $b(t)$. We will see that this leads in particular to a simple mean-reverting linear drift. This does not mean however that this is the only way to obtain a mean reverting uniform diffusion, there are many others. Indeed, it would be enough to set for example $Z_t:=2\Phi\left(\frac{W_t}{\sqrt{t}}\right)-1$ to obtain a standard uniform process, see Appendix \ref{app:sampling} and the related discussion. \rosso{We notice en passant that bounded stochastic processes received surprisingly little attention in the literature (see e.g. the Jacobi process or the $\Phi$-martingale in~\cite{gour02,vrinsjeanblanc,Carr17}).}

Define the re-scaled process
\[ Z_t= X_t / b(t) , \ \ X_t = b(t) Z_t \ \ \mbox{for} \ \ \ t\ge {t_0}, \] \mb{i.e. with the notation of the previous section, $Z_t=Y_{\gamma_t}$.}
Since for all $t>0$ the random variable $X_t$ has a  uniform law in $[-b(t),b(t)]$, $Z_t$ has a uniform law in $[-1,1]$ for all $t\ge {t_0}$.
We can derive the SDE for $Z_t$, $t \ge {t_0}$, using integration by parts and use that dynamics to define a new process $\tZ$:
\[ d \tZ_t = - \frac{\dot{b}(t)}{b(t)} \tZ_t dt + \left( \frac{\dot{b}(t)}{b(t)}(1-\tZ_t^2)    \right)^{1/2} \ind_{\{\tZ_t\in [-1,1]\}}  dW_t, \ \ t \ge {t_0},\ \   \tZ_{t_0} := \zeta \sim U([-1,1]). \]
Thus, with this deterministic re-scaling, we have a process $Z$ with fixed uniform distribution and fixed boundaries.
Here we assume the initial condition $\zeta$ to be independent of the driving Brownian motion.

If instead we aim to obtain a standard uniform in $[0,1]$, we adopt a slightly different transformation:
\[ \bar{Y}_t= (X_t / b(t) +1)/2 = X_t / (2 b(t)) +1/2 \]
from which
\[ X_t = 2 b (\bar{Y}_t - 1/2) .\]
By Leibnitz's rule
we have the following
\begin{theorem}
Assumptions on $b$ as in Theorem \ref{th:exist} but extended to all $T$:
let $b$ be a strictly increasing function defined on $[0,+\infty)$, continuous  and continuously differentiable in $(0,+\infty)$.  Assume  $b(0)=0$. Assume  $\dot{b}$ to be bounded in $(0,T]$ for all $T>0$. Assume further that $\lim_{t\uparrow +\infty} b(t) = +\infty$.  Consider, for $t \ge {t_0}$, the SDEs
\[ d\bar{Y}_t  = \frac{\dot{b}(t)}{ b(t)} (1/2- \bar{Y}_t) dt+ \frac{1}{2 b}\left(\ind_{\{\bar{Y}_t \in (0,1)\}} b(t)\dot{b}(t) (1 -
4( \bar{Y}_t - 1/2)^2)\right)^{1/2} dW_t, \ \ \bar{Y}_{t_0} = \xi \sim U([0,1]) \]
and
\begin{equation}\label{eq:unifZ} d \widetilde{Z}_t = - \frac{\dot{b}(t)}{b(t)} \widetilde{Z}_t dt + \left( \frac{\dot{b}(t)}{b(t)}(1-\widetilde{Z}_t^2)    \right)^{1/2} \ind_{\{\widetilde{Z}_t\in [-1,1]\}}  dW_t, \ \ \widetilde{Z}_{t_0} = \zeta \sim U([-1,1])
\end{equation}
with $\xi$ and $\zeta$ independent of $W$.
The unique solution of these SDEs mean-revert to $1/2$ and $0$ respectively with reversion speed (defined as minus the drift rate)  $\dot{b}/b$ and are distributed at any point in time as a standard uniform random variable and as a uniform $[-1,1]$ random variable respectively.
\end{theorem}
\begin{proof}
The proof is immediate. For the mean reverting behaviour, taking for example $\widetilde{Z}$, we note that $\lim_{t \uparrow +\infty} \mathbb{E}[\widetilde{Z}_t] =0$, and
$\lim_{t \uparrow +\infty} \mathbb{V}ar[\widetilde{Z}_t] = 1/3$. Actually, we are in a special case where mean and variance are constant. Furthermore, whenever $\widetilde{Z}_t$ is above the long term mean $0$, the drift is negative, pointing back to $0$, while the variance remains bounded. A similar symmetric pattern is observed when   $\widetilde{Z}_t$ is below zero.
We can further compute
\[ \lim_{t \uparrow +\infty} \mathbb{E}[\widetilde{Z}_t|\widetilde{Z}_s = z] = \lim_{t \uparrow +\infty} z \exp\left(-\int_s^t
\frac{\dot{b}(u)}{b(u)} du\right) =
\lim_{t \uparrow +\infty} z \exp\left(-\int_s^t
 d \ln b(u) \right) =
 \lim_{t \uparrow +\infty} z \frac{b(s)}{b(t)} = 0.
\]
We thus see that after a sufficiently long time the value $z$ at time $s$ is forgotten by the mean.
\end{proof}

\begin{remark} Note that mean reversion holds also under the weaker assumptions of Theorem \ref{th:existnouniq} similarly extended to $(0,+\infty)$, provided that again $\lim_{t \uparrow +\infty} b(t)= +\infty$. This is the case for example with $t^\alpha$ with $\alpha \in [1/2,1)$.
\end{remark}

We have shown above that mean reversion holds. In fact, we can say more than this, and we now analyze the limit behaviour of the process law and its exact transition densities.
\section{Exact transition densities, limit laws and ergodic properties}\label{sec:ergodicth}

We now study the transition densities and the limit laws of the process $Z$.

\subsection{The special case $b(t) = b_0 \exp(k t)$}\label{sec:ergodicspecial}

In the special case $b(t) = b_0 \exp(k t)$ with $b_0>0$ we need $X$ starting with\\  $X_{t_0} \sim U([-b_0 \exp(k t_0) ,b_0\exp(k t_0)])$. In this case we could also take $t_0=0$ since there is no singularity at time $0$. The setting is slightly different than our earlier setting because even with $t_0=0$ the cone would not start with a point but rather with the interval $[-b_0,b_0]$. In particular, the initial condition for $X$ would not be $X_0=0$; instead, $X_{0}$ would be requested to be a random variable with uniform law in $[-b_0,b_0]$.  In this case we have the special property that

\[ \dot{b}(t)/b(t) = k \]
is constant and the general SDE
\begin{equation}\label{eq:unifZ2} d \widetilde{Z}_t = - \frac{\dot{b}(t)}{b(t)} \widetilde{Z}_t dt + \left( \frac{\dot{b}(t)}{b(t)}(1-\widetilde{Z}_t^2)    \right)^{1/2} \ind_{\{\widetilde{Z}_t\in [-1,1]\}}  dW_t, \ \ \widetilde{Z}_{t_0} = \zeta \sim U([-1,1])
\end{equation}
is in fact a {\emph{time homogeneous diffusion}}
\begin{equation}\label{eq:unifZ3} d \widetilde{Z}_t = - k  \widetilde{Z}_t dt + \left( k (1-\widetilde{Z}_t^2)    \right)^{1/2} \ind_{\{\widetilde{Z}_t\in [-1,1]\}}  dW_t, \ \ \widetilde{Z}_{t_0} = \zeta \sim U([-1,1])
\end{equation}
to which we can apply standard boundary and ergodic theory techniques for time homogeneous one-dimensional diffusions, see for example \cite{karlin}.

Let's analyze Eq \eqref{eq:unifZ3} using the standard theory. First of all in this case we already know from our previous analysis of $X$ that, if $\bar{p}$ is the density of a $U([-1,1])$ random variable then $\bar{p}$ satisfies the Fokker Planck equation for the marginal density of the diffusion \eqref{eq:unifZ3} so that
\[ {\cal L}^* \bar{p} = 0 ,\]
where ${\cal L}^*$ is the forward diffusion operator of the Fokker Planck equation.
This means that $\bar{p}$ is the invariant measure for the diffusion \eqref{eq:unifZ3}.

This can be further confirmed by the standard calculation: given a diffusion process with drift $\mu$ and diffusion coefficient $\sigma$, under suitable conditions (see for example \cite{locherbach}) the invariant measure is proportional to

\[ \frac{2}{\sigma^2(x) \exp \left(-2 \int_{x_0}^x \frac{\mu(u)}{\sigma^2(u)} du\right)} \]
which, with our $\mu(x) = -k x$ and $\sigma(x) = \left( k (1-x^2)    \right)^{1/2}$ results immediately in a uniform density. Hence we have that the uniform is the invariant measure of our diffusion and that our diffusion is ergodic. We also have
\[ \lim_{t \uparrow +\infty} p_{Z_{t+s}|Z_s}(y;x)  = \lim_{t \uparrow \infty} p_{Z_{t}|Z_0}(y;x)  = \bar{p}(y) \ \ \mbox{for all} \ \ s>0, \ \ x \in [-1,1]. \]

\subsection{The general case with curved boundary}
Now we move to the case of the full $Z$ with general boundary $b(t)$ in Eq \eqref{eq:unifZ2}.

We already know that the density $\bar{p}$ satisfies the Fokker Planck equation for the marginal density of \eqref{eq:unifZ2}.
Given that $\partial \bar{p} / \partial t =0$ and that the Fokker Planck equation reads $\partial p_t / \partial t = {\cal L}^*_t p_t$ we deduce that
\[ {\cal L}^*_t \bar{p} =0 \]
for the operator $\cal L$ of \eqref{eq:unifZ2}. Hence $\bar{p}$ is also the invariant measure for the more general case \eqref{eq:unifZ2}. It's not clear beforehand however that the diffusion \eqref{eq:unifZ2} has a limit transition law.

To check this, we first derive its exact transition laws.
We have the following

\begin{theorem}[\mb{Moments}  for the time-inhomogeneous mean-reverting uniform diffusion \eqref{eq:unifZ2}] Let $\xi_{n}:=\hbox{mod}(n,2)$ stand for the \textit{odd indicator} and $\mu_n:=(1-\xi_{n})/(n+1)$ denote the $n$-th moment of a random variable uniformly distributed in $[-1,1]$. Then, the conditional moments $M_n(s,t;z) := \E[\tilde{Z}_t^n | \tilde{Z}_s = z], t\geq s$ are given by
\[
M_n(s,t;z)=\mu_n+\sum_{k=1}^{\frac{n+\xi_{n}}{2}}(-1)^{k}(z^{2k-\xi_{n}}-\mu_{2k-\xi_{n}})\sum_{j=k}^{\frac{n+\xi_{n}}{2}}\alpha_{j,k}[n](-1)^j\left(\frac{b(s)}{b(t)}\right)^{j(2(j-\xi_{n})+1)}
\]
with $\alpha[n]$ being $\frac{n+\xi_{n}}{2}$-by-$\frac{n+\xi_{n}}{2}$ lower triangular matrices (i.e. $\alpha_{j,k}[n]=0$ for all $k>j$) whose lower entries are defined as
\begin{equation}
\alpha_{j,k}[n]=\left\{\begin{array}{ll}
1&\hbox{ if }j=k=\frac{n+\xi_{n}}{2}\\
-(-1)^{\frac{n+\xi_{n}}{2}}\sum_{i=k}^{\frac{n+\xi_{n}}{2}-1}\alpha_{i,k}[n](-1)^i&\hbox{ if }j=\frac{n+\xi_{n}}{2},\; k<j\\
\frac{\alpha_{j,k}[n-2]n(n-1)}{n(n+1)-2j(2(j-\xi_{n})+1)}&\hbox{otherwise}
\end{array}\right.
\label{eq:sumcoeff}
\end{equation}

\end{theorem}
Note: The explicit expressions for the first six moments are given in the appendix.

\begin{proof}

The dynamics of powers of $\tilde{Z}$ solving \eqref{eq:unifZ2} are easily found from Ito's formula. This yields the ODE governing the conditional expectations for all $n$. For $n=0$ one trivially has $M_0(s,t;z)=1$. Now, set $h(t):=\dot{b}(t)/b(t)$ satisfying
\[
\exp\left\{-\int_s^t h(u)du\right\}=\exp\left\{\int_t^s d\ln u\right\}=b(s)/b(t)\;.
\]

Hence,

\[
\dot{M}_1(s,t;z):=\frac{\partial M_1(s,t;z)}{\partial t}=-h(t)M_1(s,t;z) \hbox{ s.t. }M_1(s,s;z)=z
\]
which leads to $M_1(s,t;z)=z\frac{b(s)}{b(t)}$.

For $n\geq 2$, one gets a recursive first order inhomogeneous ODE

\begin{eqnarray}
\dot{M}_n(s,t;z) &=& -\frac{n(n+1)}{2} \frac{\dot{b}(t)}{b(t)} {M}_n(s,t;z) + \frac{n(n-1)}{2} \frac{\dot{b}(t)}{b(t)} {M}_{n-2}(s,t;z) \\
&=&\frac{n(n+1)}{2} h(t) \left( \frac{n-1}{n+1}  {M}_{n-2}(s,t;z)-{M}_n(s,t;z)\right)\;,\label{eq:Mnstz}
\end{eqnarray}

whose solution is

\begin{equation}
{M}_n(s,t;z) = z^n \left(\frac{b(s)}{b(t)}\right)^{n(n+1)/2} + \frac{n(n-1)}{2}\int_s^t \left(\frac{b(u)}{b(t)}\right)^{n(n+1)/2}  h(u) {M}_{n-2}(s,u;z) du\;. \label{MnstSOL}
\end{equation}

Notice that the expression above satisfies the initial conditions $M_{n}(s,s;z)=z^{n}$ for all $n$ in $2,3,\ldots$. This is also the case for the expression stated in the theorem as a result of the relationship between the entries of the $\alpha[n]$ matrices: as $b(s)/b(t)=1$ when $s=t$, the double sum collapses to the single $j=k=(n+\xi_{n})/2$ term. This concludes the check of the initial conditions.

Replacing $n$ by $n+2$ in the $M_n(s,t;z)$ expression given in the theorem yields

\begin{eqnarray}
M_{n+2}(s,t;z)&=&\mu_{n+2}+\sum_{k=1}^{\frac{n+2+\xi_{n}}{2}}(-1)^{k}(z^{2k-\xi_{n}}-\mu_{2k-\xi_{n}})\sum_{j=k}^{\frac{n+2+\xi_{n}}{2}}\alpha_{j,k}[n+2](-1)^j\left(\frac{b(s)}{b(t)}\right)^{j(2(j-\xi_{n})+1)}\\
&=&\mu_{n+2}+\sum_{k=1}^{\frac{n+\xi_{n}}{2}}(-1)^{k}(z^{2k-\xi_{n}}-\mu_{2k-\xi_{n}})\sum_{j=k}^{\frac{n+2+\xi_{n}}{2}}\alpha_{j,k}[n+2](-1)^j\left(\frac{b(s)}{b(t)}\right)^{j(2(j-\xi_{n})+1)}\nonumber\\
&&+(-1)^{\frac{n+2+\xi_{n}}{2}}(z^{n+2}-\mu_{n+2})\alpha_{\frac{n+2+\xi_{n}}{2},\frac{n+2+\xi_{n}}{2}}[n+2](-1)^{\frac{n+2+\xi_{n}}{2}}\left(\frac{b(s)}{b(t)}\right)^{\frac{n+2+\xi_{n}}{2}(n+3-\xi_{n})}\\
&=&\mu_{n+2}+(z^{n+2}-\mu_{n+2})\left(\frac{b(s)}{b(t)}\right)^{\frac{n+2}{2}(n+3)}+I_1+I_2\nonumber\\
I_1&:=&\sum_{k=1}^{\frac{n+\xi_{n}}{2}}(-1)^{k}(z^{2k-\xi_{n}}-\mu_{2k-\xi_{n}})\sum_{j=k}^{\frac{n+\xi_{n}}{2}}\alpha_{j,k}[n+2](-1)^j\left(\frac{b(s)}{b(t)}\right)^{j(2(j-\xi_{n})+1)}\\
I_2&:=&(-1)^{\frac{n+2+\xi_{n}}{2}}\left(\frac{b(s)}{b(t)}\right)^{\frac{n+2}{2}(n+3)}\sum_{k=1}^{\frac{n+\xi_{n}}{2}}(-1)^{k}(z^{2k-\xi_{n}}-\mu_{2k-\xi_{n}})\alpha_{\frac{n+2+\xi_{n}}{2},k}[n+2]
\end{eqnarray}

where we have used $\xi_{n}=\xi_{n}^2$ and $\alpha_{\frac{n+2+\xi_{n}}{2},\frac{n+2+\xi_{n}}{2}}[n+2]=1$ from \eqref{eq:sumcoeff} with $n\leftarrow n+2$.

It remains to check that this expression agrees with the solution \eqref{MnstSOL} when setting $n\leftarrow n+2$. The constant term trivially reads

\[
z^{n+2}\left(\frac{b(s)}{b(t)}\right)^{\frac{n+2}{2}(n+3)}\;.
\]

The integral can be split in two parts with respect to $M_n$. The first part of $M_n$ is $\mu_n$ and the second is the double sum. The first part is
\begin{eqnarray}
\frac{(n+2)(n+1)}{2}\int_s^t \left(\frac{b(u)}{b(t)}\right)^{(n+2)(n+3)/2}  h(u) \mu_n du  &=& \mu_n\frac{n+1}{n+3}\left(1-\left(\frac{b(s)}{b(t)}\right)^{(n+2)(n+3)/2}\right)\nonumber\\
&=& \mu_{n+2}\left(1-\left(\frac{b(s)}{b(t)}\right)^{(n+2)(n+3)/2}\right)
\end{eqnarray}

It remains to show that the remaining integral agrees with $I_1+I_2$ defined above. It comes

\[
\alpha_{jk}[n]\frac{(n+2)(n+1)}{2}\int_s^t \left(\frac{b(u)}{b(t)}\right)^{(n+2)(n+3)/2}  h(u) \left(\frac{b(s)}{b(u)}\right)^{j(2(j-\xi_{n})+1)} du
=J_1(j,k,n)+J_2(j,k,n)
\]

where

\begin{eqnarray}
J_1(j,k,n)&:=&\alpha_{j,k}[n+2]\left(\frac{b(s)}{b(t)}\right)^{j(2(j-\xi_{n})-1)}\;,\nonumber\\
J_2(j,k,n)&:=&-\alpha_{j,k}[n+2]\left(\frac{b(s)}{b(t)}\right)^{(n+2)(n+3)/2}\;.\nonumber
\end{eqnarray}

It is easy to see that

\[
\sum_{k=1}^{\frac{n+\xi_{n}}{2}}(-1)^{k}(z^{2k-\xi_{n}}-\mu_{2k-\xi_{n}})\sum_{j=k}^{\frac{n+\xi_{n}}{2}}(-1)^jJ_1(j,k,n)=I_1 .
\]

On the other hand,

\begin{eqnarray}
&&\sum_{k=1}^{\frac{n+\xi_{n}}{2}}(-1)^{k}(z^{2k-\xi_{n}}-\mu_{2k-\xi_{n}})\sum_{j=k}^{\frac{n+\xi_{n}}{2}}(-1)^jJ_2(j,k,n)\nonumber\\
&=&-\left(\frac{b(s)}{b(t)}\right)^{(n+2)(n+3)/2}\sum_{k=1}^{\frac{n+\xi_{n}}{2}}(-1)^{k}(z^{2k-\xi_{n}}-\mu_{2k-\xi_{n}})\sum_{j=k}^{\frac{n+\xi_{n}}{2}}(-1)^j\alpha_{j,k}[n+2]\nonumber\\
&=&(-1)^{\frac{n+2+\xi_{n}}{2}}\left(\frac{b(s)}{b(t)}\right)^{(n+2)(n+3)/2}\sum_{k=1}^{\frac{n+\xi_{n}}{2}}(-1)^{k}(z^{2k-\xi_{n}}-\mu_{2k-\xi_{n}})\alpha_{\frac{n+2+\xi_{n}}{2},k}[n+2]
\end{eqnarray}

where the last inequality results from \eqref{eq:sumcoeff} with $n\leftarrow n+2$; this is nothing but $I_2$. This completes the proof.
\end{proof}

\begin{corollary}[Limit law for the transition densities of \eqref{eq:unifZ2}] When $b$ is grounded and non-decreasing, the solution of the SDE \eqref{eq:unifZ2} conditional on $\tilde{Z}_s=z\in[-1,1]$, $s\geq 0$ admits a stationary law in the sense that each conditional moment of the solution tends to a constant. If, moreover, $\lim_{t\to\infty} 1/b(t)=0$ then then stationary law is ${U}(-1,1)$.
\end{corollary}

We finally confirm the intuition given the above moments result, showing that we can connect the general case to the special time-homogeneous case discussed in Section \ref{sec:ergodicspecial}. To do this, it will be enough to introduce a deterministic time change. The following proposition is essentially equivalent to the methodology in \cite{peacock} that we already used in the proof of Theorem \ref{th:existnouniq}, but given the different context we state and prove the proposition explicitly for convenience.

\begin{proposition}{[General mean reverting SDE as a time--changed time homogeneous SDE]}
Consider the general SDE \eqref{eq:unifZ2} for $\tilde{Z}_t$ with $t \ge t_0$. There exists a Brownian motion $B$ such that
\[   \tilde{Z}_t = \xi_{\tau(t)}  \ \ \ \mbox{for the deterministic time change} \ \ \tau(t) = \ln(b(t)),
\]
where $\xi$ is the solution of the following SDE driven by $B$:
\begin{equation} d \xi_t = - \xi_t\ dt + \left(1-\xi_t^2\right)^{1/2} 1_{\{\xi_t \in [-1,1]\}} dB_t ,
\end{equation}
 provided that
\[t_0 = b^{-1}(1), \ \     \xi_{\tau(t_0)} = \xi_0 = \tilde{Z}_{t_0} \]
and that the initial condition is assumed to be a random variable $\xi_{0}$ with uniform law in $[-1,1]$ and independent of $B$.
\end{proposition}
\begin{proof}
Consider the SDE \eqref{eq:unifZ2} for $\tilde{Z}$. This is driven by the continuous martingale
\[ M_t = \int_{0}^t \left(\frac{\dot{b}(s)}{b(s)}\right)^{1/2} dW_s  \]
in that it can be written as
\[ d \widetilde{Z}_t = -\frac{\dot{b}(t)}{b(t)} \widetilde{Z}_t dt +
\left(1-\widetilde{Z}_t^2\right)^{1/2} 1_{\{Z_t \in [-1,1] \}} dM_t .\]
Note that the quadratic variation of $M$ is given by $\langle M \rangle_t = \tau(t)$.
From the Dambis, Dubins--Schwarz (DDS) theorem we know that there exists a Brownian motion $B$ such that
\[ M_t = B_{\langle M \rangle_t} = B_{\tau(t)}.\]
If we further notice that $d \tau(t) = (\dot{b}(t)/b(t)) dt$ we can write SDE \eqref{eq:unifZ2} as
\[ d \widetilde{Z}_t =  - \widetilde{Z}_t\ d \tau(t) +
\left( 1-\widetilde{Z}_t^2 \right)^{1/2} 1_{\{\widetilde{Z}_t \in [-1,1]  \}} dB_{\tau(t)} \]
so that if we set $\xi_{\tau(t)} :=\widetilde{Z}_t$ and substitute in the last SDE above we conclude.
\end{proof}

The assumption that $t_0 = b^{-1}(1)$ (we could also take a larger $t_0$) is needed to avoid negative time in the $\xi$ SDE, but this is not an issue since we are interested in the limiting behaviour of the solution for the SDE of $\tilde{Z}_t$ for large $t$.

Given our discussion in Section \ref{sec:ergodicspecial}, we know that $\xi$ is ergodic and has a uniform invariant measure as limit law. We can then confirm our earlier result on the limit law of $\tilde{Z}$: it will be a uniform law that forgets the initial condition at an earlier time, and the $\tilde{Z}$ process will be a deterministic time-change of an  ergodic process.

\begin{figure}
\centering
\subfigure[$z=-0.95,b(t)=2t^{3/2}$]{\includegraphics[width=0.48\columnwidth]{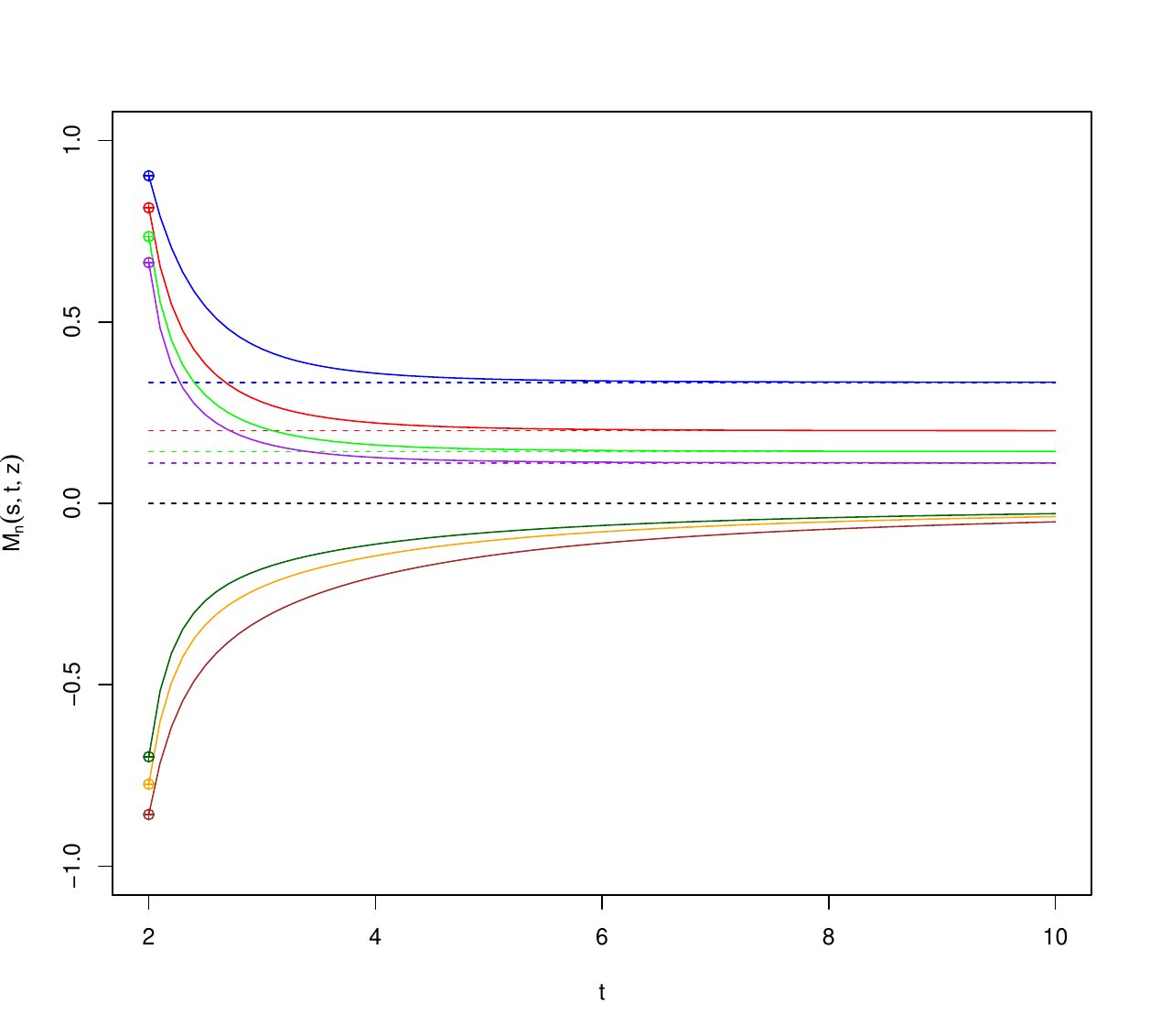}}\hspace{0.2cm}
\subfigure[$z=0.95,b(t)=2t^{3/2}$]{\includegraphics[width=0.48\columnwidth]{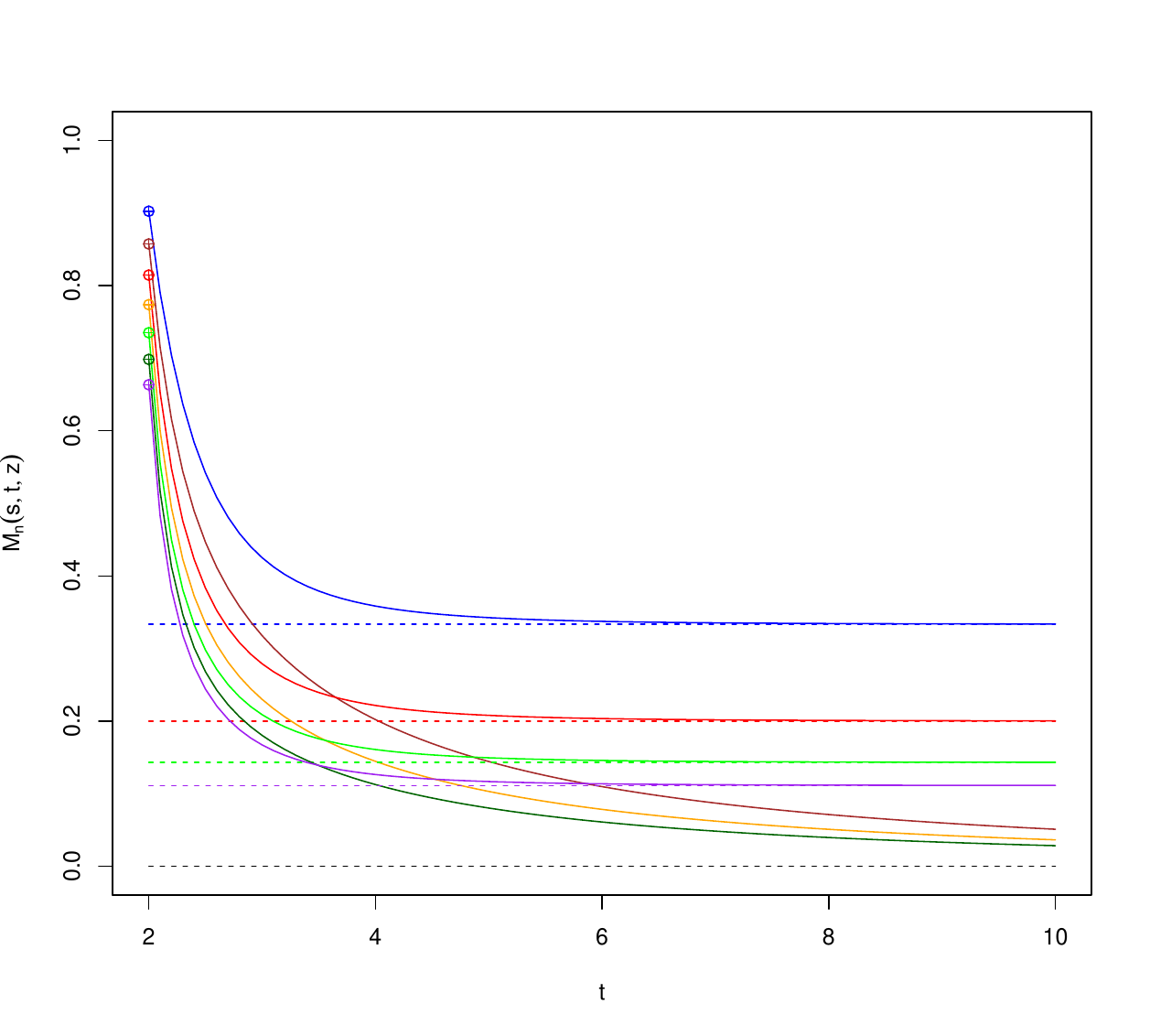}}\\
\subfigure[$z=-0.5,b(t)=2t^{3/2}$]{\includegraphics[width=0.48\columnwidth]{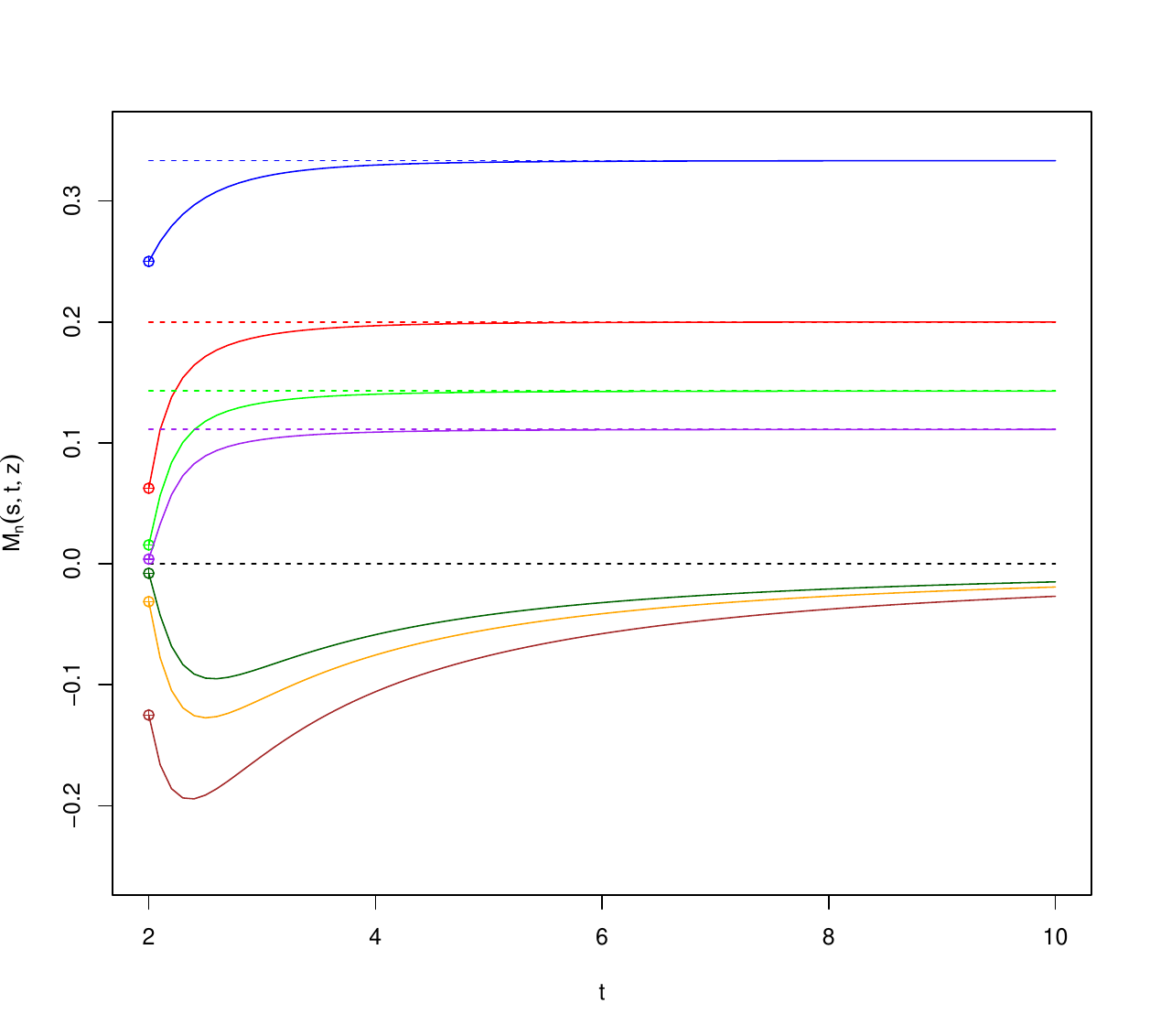}}\hspace{0.2cm}
\subfigure[$z=-0.85,b(t)=1-e^{-t/3}$]{\includegraphics[width=0.48\columnwidth]{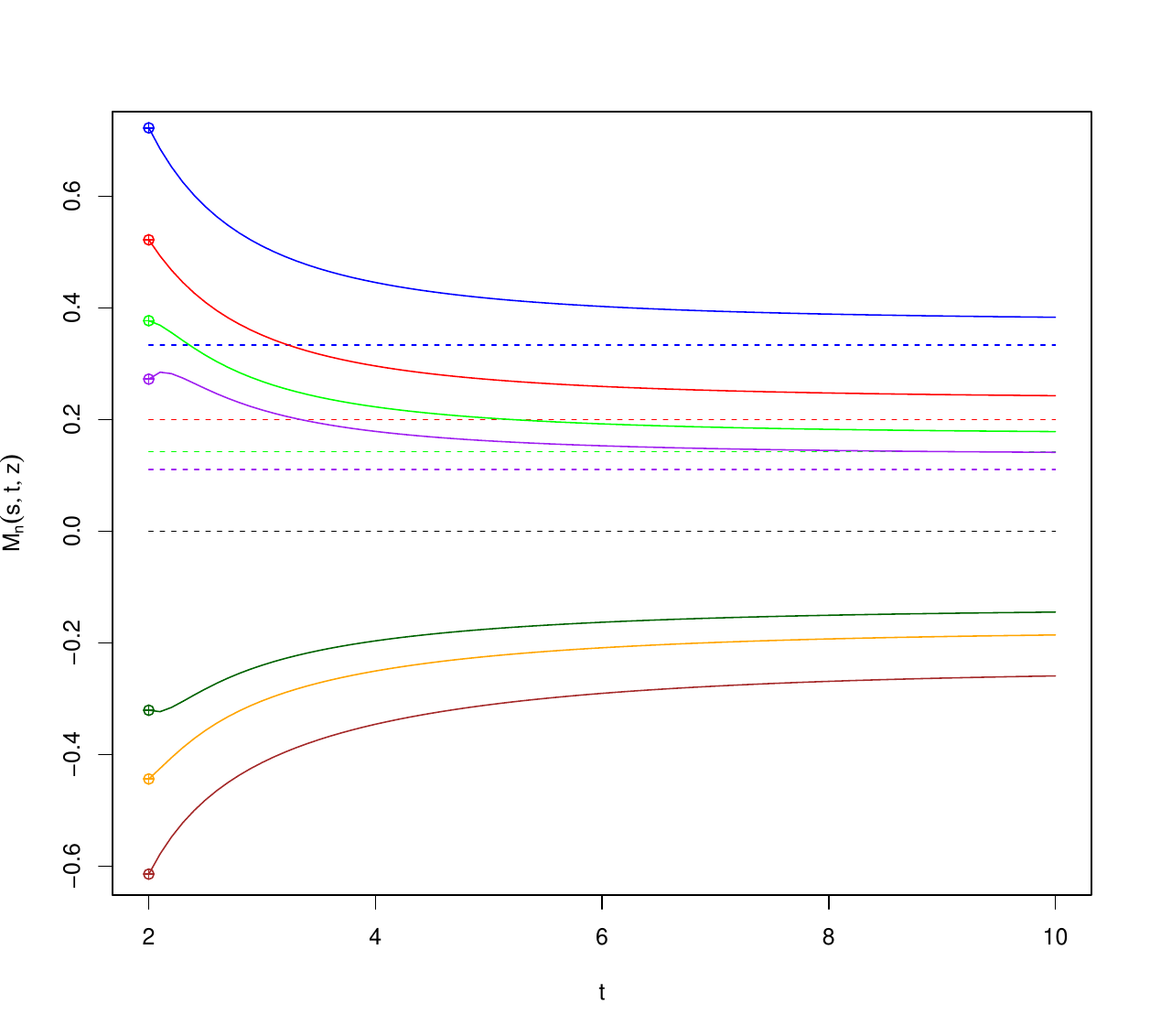}}
\caption{Evolution of the $2,\ldots,8$-th conditional moments for $s=2$ up to $T=10$ in the following order: blue, brown, red, orange, green, dark green, purple.}\label{fig:hist1}
\end{figure}

\section{Local time at the boundaries and potential applications}\label{sec:localtime}

We now discuss the behaviour of the solution $\widetilde{Z}$ of \eqref{eq:unifZ} at the boundaries -1 and 1, and thus the behaviour of the original $X_t$, solution of \eqref{eq:martinX2}, at the boundaries $-b(t)$ and $b(t)$.

\begin{theorem}{[Local time calculation.]}\label{th:attain}

Given a strictly increasing function $b$ defined in $[0,T]$, continuous, and differentiable in $(0,T]$, assume $b(0)=0$ and $\dot{b}\ b$ to be bounded in $(0,T]$, with finite limit $\lim_{t \downarrow 0} \dot{b}(t)b(t)$ (this holds under the assumptions of Theorem \ref{th:exist} and is a slight reinforcement of the assumptions of Theorem \ref{th:existnouniq}). The local time for the  process $b(t)-X_t$  (resp. $X_t+b(t)$ ) at level $0$  is zero.
\end{theorem}

\begin{proof}
Let us introduce $U_t= b(t)-X_t$.  Then $$d \langle U\rangle _t= \ind_{\{0\leq U_t \leq 2b(t)\}} \left(\frac{\dot b(t)}{b(t)}\right) U_t (2b(t)-U_t)dt$$ Then
\begin{eqnarray} t  &\geq &  \int_0^t   \ind_{\{0\leq U_s \leq 2b(s)\}} ds= \int_0^t  \ind_{\{0\leq U_s \leq 2b(s)\}} \frac{  b(s)}{ \dot b(s)} \frac{1}{(2b(s)-U_s) U_s}d \langle U\rangle _s \nonumber\\
&=&   \int_0^\infty  da  \int_0^t    \ind_{\{0\leq a \leq 2b(s)} \frac{  b(s)}{ \dot b(s)} \frac{1}{(2b(s)-a) a}   d_sL^a_s     \nonumber \end{eqnarray}
where the last equality comes from an extension of the occupation time formula
(\cite{ry:cm}, Chapter VI, Section 1, Corollary 1.6) as in \cite{peskir}.
%

We note that $b(s)/\dot{b}(s)$ is bounded from below by a positive constant $C$ for all $s \ge \delta $. We can easily see that this is indeed the case since $\dot{b}(s) {b}(s)$ is bounded by above in $[0,T]$ by assumption, say by a constant $K>0$, so that $\dot{b}(s)/b(s) = \dot{b}(s) b(s) / b(s)^2 \le K / b(\delta)^2 =: C$. This implies that
$b(s)/\dot{b}(s) \ge C$ for all $t \ge \delta$.
\\
We obtain  $$t  \geq  C \int_0^\infty  da \int_\delta^t  \ind_{\{0\leq a \leq 2b(\delta)\}}     \frac{dL^a_s}{(2b(\delta)-a)a}   \geq  C \int_0^{2b(\delta)} \frac{L^a_t-L^a_\delta}{(2b(\delta)-a)a}da $$ which implies that $L^0_t-L^0_\delta=0$. By continuity, $L^0_\delta$ goes to $0$ when $\delta$ goes to $0$.
\end{proof}

We conclude this section with a hint at potential applications of our processes and with two remarks.
$\bar{Y}$ can be used for example to model the dynamics of recovery rates  or probabilities in the case of no information (maximum entropy), whereas $\widetilde{Z}$ can be used as a model for stochastic correlation.

\begin{remark} The above construction for $\bar{Y}$ and $\widetilde{Z}$, mean-reverting uniform diffusions with fixed boundaries based on rescaling the process $X$ of  Theorem \ref{th:exist}, has the drawback of starting time at ${t_0}>0$, without defining the dynamics in $[0,{t_0})$. This is done to  avoid singularities in $t=0$ with the rescaling.  On the other hand, it has the advantage that the solution is unique in the strong sense.
An alternative for obtaining a similar process, especially for cases like $b(t) = \sqrt{t}$, is to start from $X$ constructed as in Theorem \ref{th:existnouniq}, requiring assumptions on $b$ that are weaker than  in  Theorem \ref{th:exist}. If we do so, and recalling  $Y$ in the proof of Theorem \ref{th:existnouniq} and Eq. \eqref{eq:XYgamma} in particular, we obviously could have $Z_t=Y_{\gamma(t)}$ where $\gamma(t)=2\ln b(t)$, or even $Z_t = Y_t$. Notice however that to get a diffusion with uniform law in $[-1,1]$ we could directly define a process $\hat{Z}$ as $\hat{Z}_t:=Y_{\alpha(t)}$ for any time change function $\alpha$ provided that it is increasing. Indeed, this would not affect the marginals of $\hat{Z}$ as $Y$ is a diffusion with uniform marginals in $[-1,1]$ at all times.
\end{remark}

\begin{remark} The above rescaling approach yields a diffusion associated to the uniform peacock with constant boundaries $-1,1$. It is also obvious from \eqref{ygamma} that defining $Z$ as $Z_t=Y_{\gamma(t)}$ will lead to a mean-reverting diffusion. However, this is a mean-reverting diffusion process and not a diffusion martingale. Still, we know since~\cite{Kell72} that there is a martingale associated to any peacock. Hence a natural question is: what is the diffusion martingale associated with this peacock ? Looking at the forward Kolmogorov equation, the answer turns out to be: only the trivial martingale diffusions with zero drift and zero diffusion coefficients. Indeed,  forcing $\varphi(x,t)$ to be the density of a uniform with fixed boundaries at all time implies that the left hand side of \eqref{eq:fokkerrho3} vanishes, leading to $\sigma(t,x)=0$ for all $x$. In other words, the only diffusion martingale associated to this peacock is the trivial martingale $Z_t=\zeta$  for all $t$, where $\zeta\sim U([-1,1])$.
\end{remark}

Finally, with a slight abuse of notation, we will denote $\bar{Y}$ by $Y$ and $\widetilde{Z}$ by $Z$ in the rest of the paper.

\section{Specific choices of the boundary $b(t)$ and links with peacocks}\label{sec:specific}
In this section we present a number of qualitatively different choices for $b(t)$.

\subsection{The square-root case $b(t)=\sqrt{t}$}

As we pointed out earlier, the case $b(t)=\sqrt t$ for \eqref{eq:martinX2}, which leads to  \[ d X_t=  \frac{1}{\sqrt 2}\sqrt{1-\frac{X_t^2}{t}}  \ind_{\{X_t \in [-\sqrt{t},\sqrt{t}]\}}dW_t, \ \ X_0=0\] corresponds exactly to the solution presented in \cite{peacock}.

\subsection{The linear case $b(t)= k t$: numerical examples and activity}
The case $b(t)=k t$ fits the assumption of Theorem \ref{th:exist} since  $\dot{b}(t)=k$ is bounded on $[0,T]$ for any $T\in\mathbb{R}^+$. Notice also that $\dot{b}(t) b(t) = k^2 t$ vanishes for $t \downarrow 0$. Our previous SDEs for $X$ \eqref{eq:martinX2} and $\widetilde{Z}$ \eqref{eq:unifZ} specialize to
\begin{equation}\label{eq:martinX2kt} d X_t = \ind_{\{X_t \in [-k t, k t]\}} \frac{1}{\sqrt{t}} \ \sqrt{ (kt)^2 - X_t^2} \ dW_t, \ \ X_0=0, \ \ X_t \sim U([-kt, kt]) \mbox{\ \ for all \ \  } t>0
\end{equation}
and
\begin{equation}\label{eq:Zkt} d Z_t = - \frac{1}{t}\ Z_t\ dt + \ind_{\{Z_t\in [-1,1]\}}\ \frac{1}{\sqrt{t}}\ \sqrt{1-Z_t^2} \  dW_t, \ \ Z_{t_0} = \zeta \sim   U([-1,1]) \mbox{\ \ for all \ \  } t\ge {t_0} .
\end{equation}
As a numerical example we implement the Euler scheme for $X$. We know from \cite{gyongy} that under our assumptions the Euler scheme converges in probability.
We thus implement a Euler scheme for the SDE for $X$ and then plot a histogram of the density. This is shown in Figure \ref{fig:hist1}.
Moreover, we show in the right panel of Figure \ref{fig:paths1} a few sample paths of the process $X$.

\begin{figure}
\centering
\includegraphics[width=0.45\columnwidth]{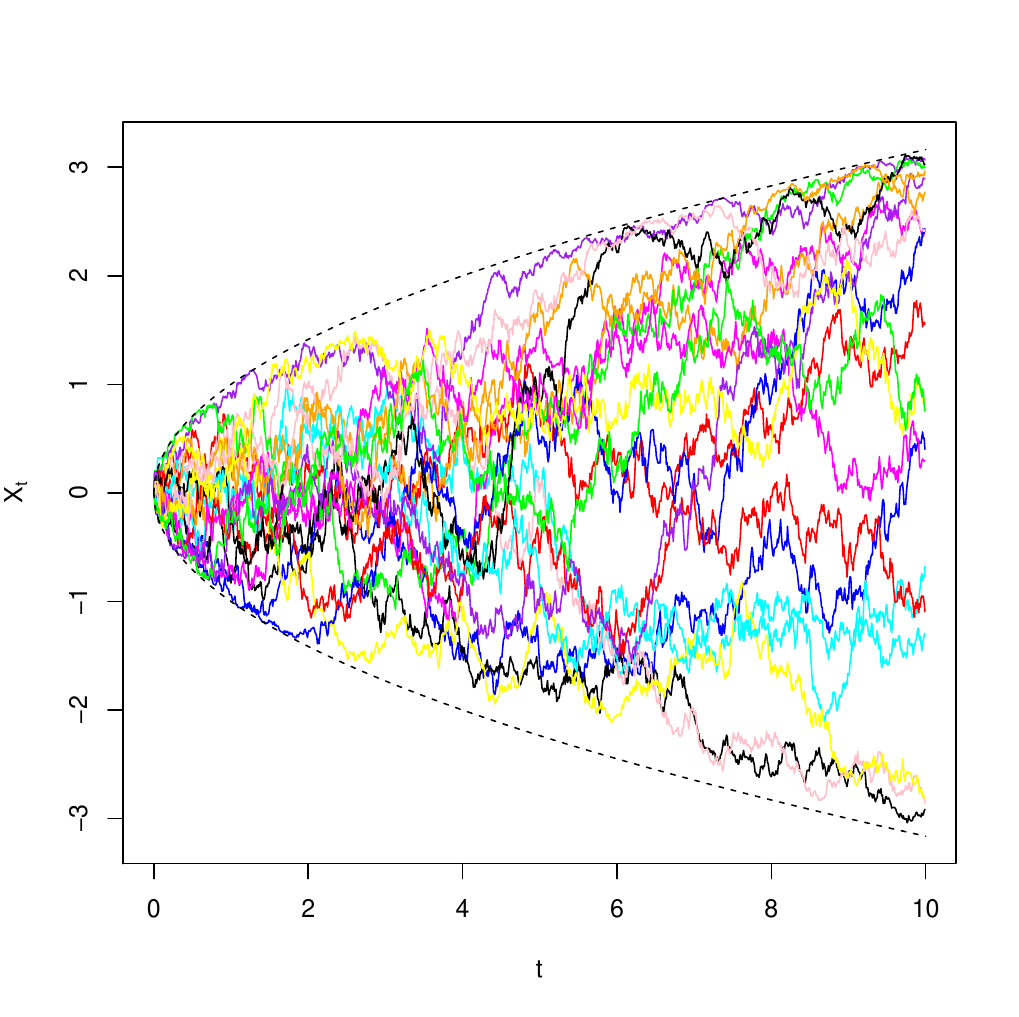}\hspace{0.2cm}
\includegraphics[width=0.45\columnwidth]{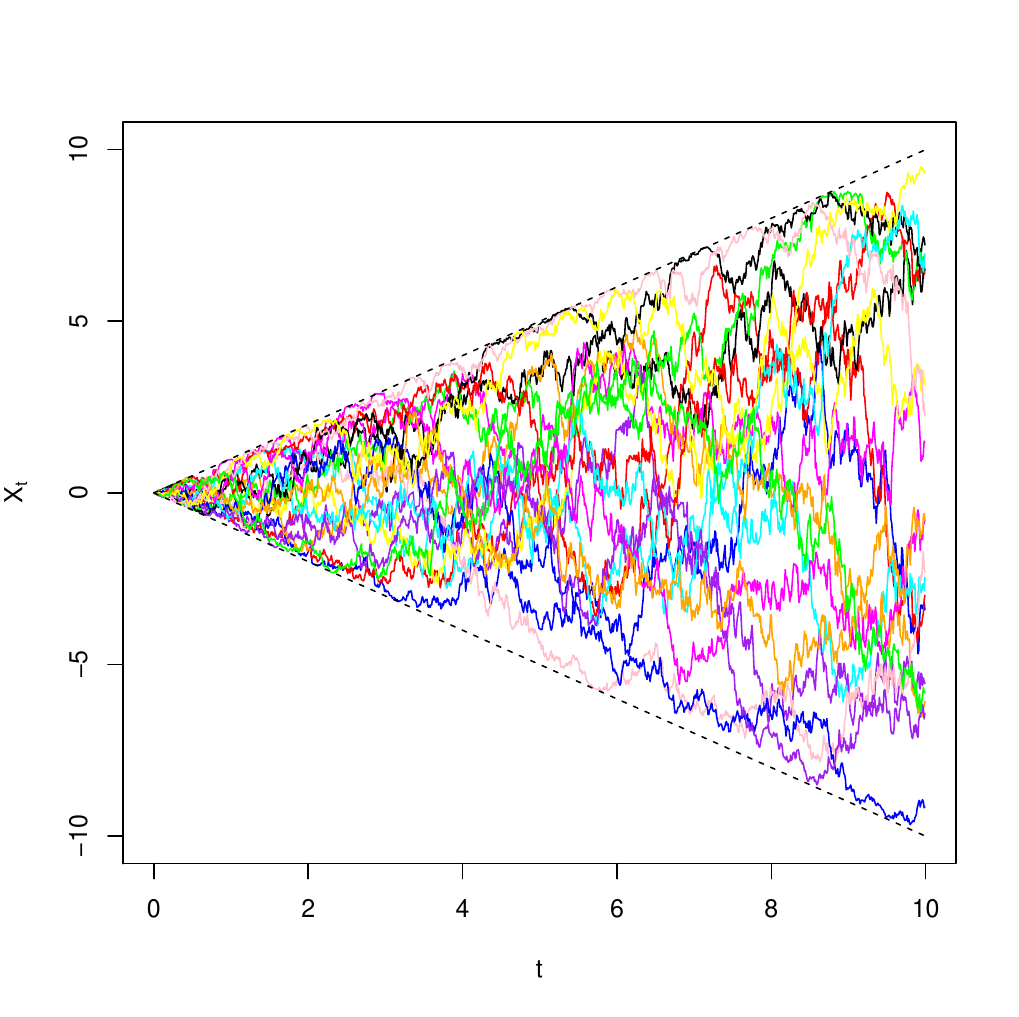}
\caption{20 paths of the SDE \eqref{eq:martinX2kt} at time $1y$ with $b(t)=\sqrt{t}$ (left) and $b(t)= t$ (right). Time step is 0.01 years. Euler Scheme.}\label{fig:paths1}
\end{figure}
\begin{figure}
\centering
\includegraphics[width=0.45\columnwidth]{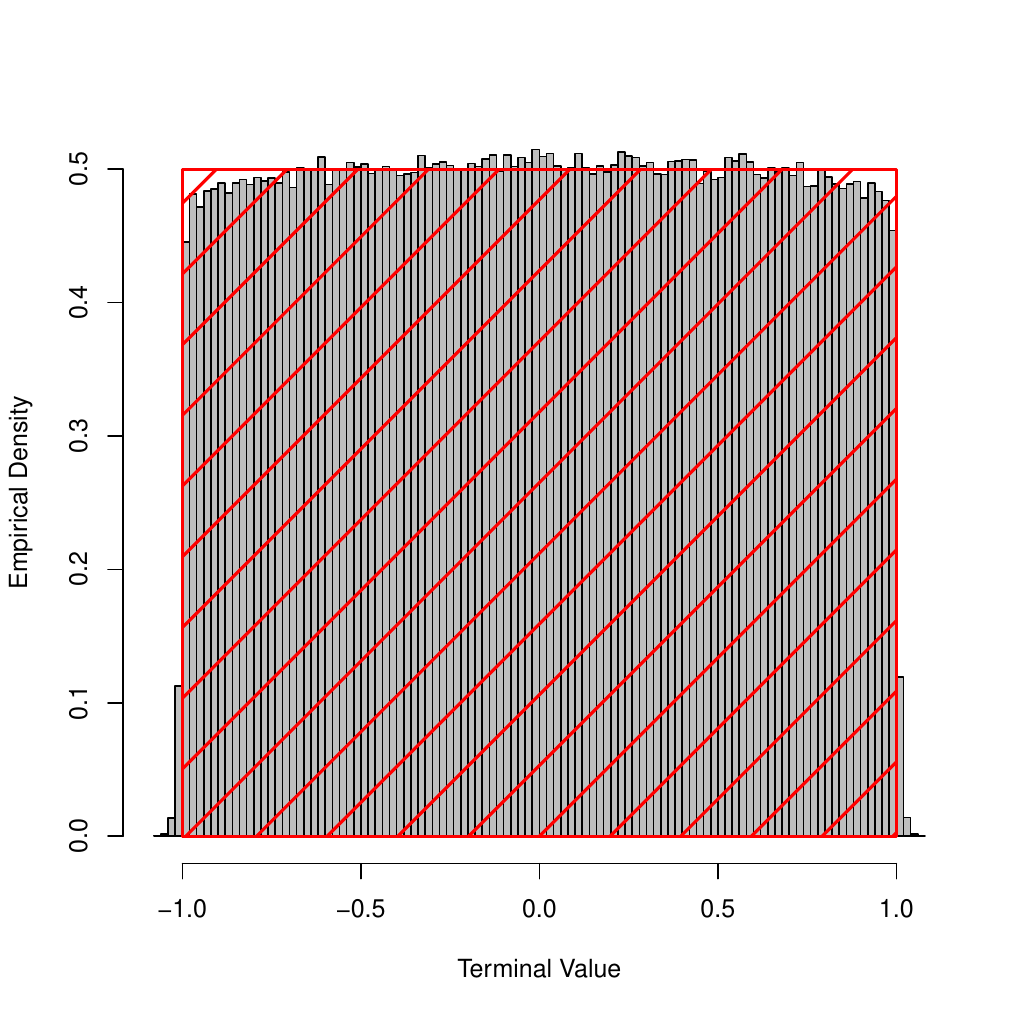}\hspace{0.2cm}
\includegraphics[width=0.45\columnwidth]{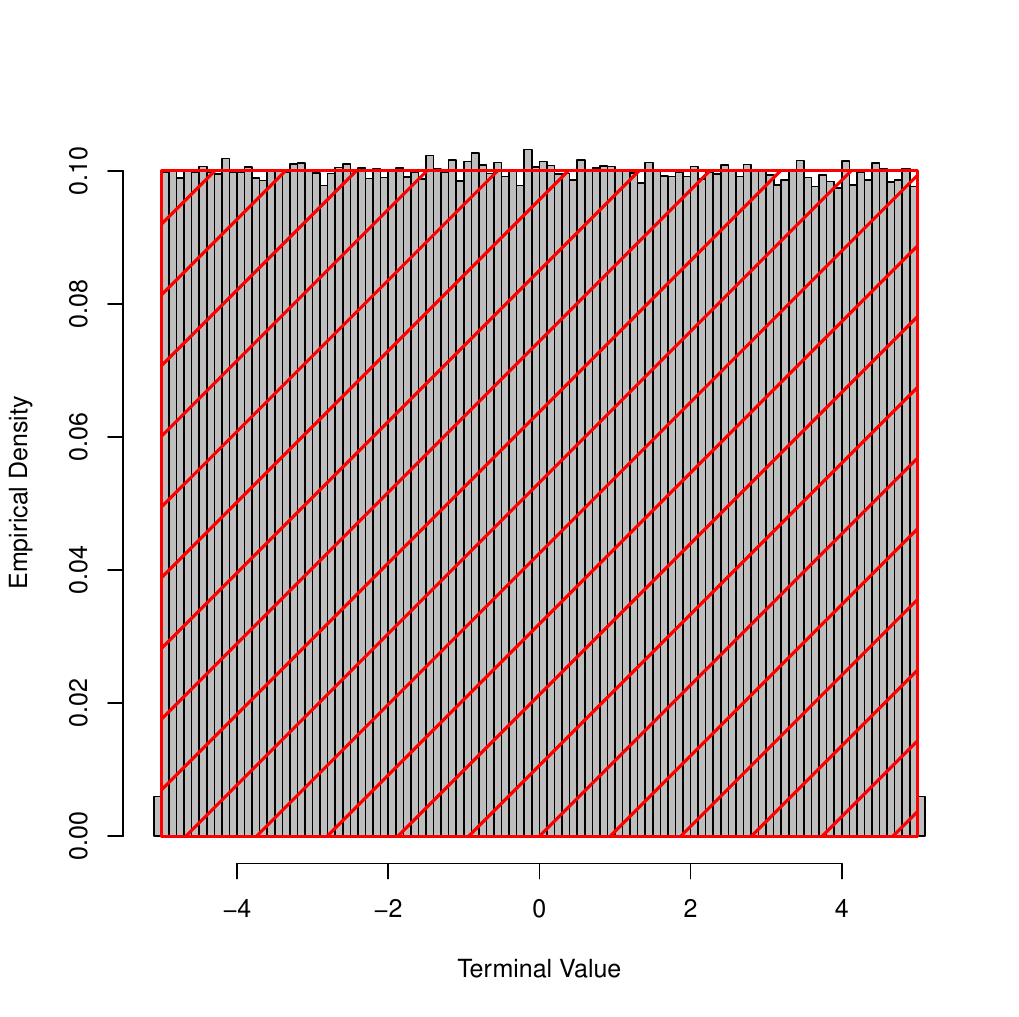}
\caption{Histograms with 100 bins for the density of the SDE \eqref{eq:martinX2kt} at time $t$ with $k=1$ for 1 Million scenarios via Euler scheme: left $t=1y$ and right $5y$. Time step is 0.01 years.}\label{fig:hist1}
\end{figure}

We may also apply Theorem \ref{th:attain} to this particular case, to see that $Z_t = X_t/(k t)$ for $t \ge t_0 >0$ spends zero time at the boundaries $-1$ and $1$. As a consequence, $X_t$ spends zero time at the boundaries $-k t$ and $k t$.

More qualitatively, we observe that $Z$ in \eqref{eq:Zkt} mean-reverts to $0$ with speed $1/t$. The speed will be very large for small time but will become almost zero when time is large. The diffusion coefficient, similarly, is divided by $\sqrt{t}$, so it will tend to vanish for large $t$.  This is confirmed by the following activity calculation. We may conclude that the process will not be absorbed in the boundary and will tend to ``slow down'' in time, while maintaining a uniform distribution.

We show that the pathwise activity of the uniform $(-1,1)$ process $Z$ is vanishing for large $t$ in the sense that the deviation of $Z_{t+\delta}(\omega)$ from $Z_{t}(\omega)$ collapses to zero for all $\delta>0$, all $\omega\in\Omega$ (the sample space) as $t\to\infty$.

\begin{lemma}

\beq
\forall \delta>0\;,~\var(Z_{t+\delta}-Z_t)\to 0~\mathrm{ as }~t\to\infty\;.\nonumber
\eeq
\end{lemma}

\begin{proof}
Notice that for all $t>0$, $\E(Z_t)=0$ so that $v:=\var(Z_t)=\E(Z^2_t)=1/3$ is the variance of a zero-mean uniform random variable distributed on $[-1,1]$. Then,

\beq
\var(Z_{t+\delta}-Z_{t})=\var(Z_{t+\delta}^2)+\var(Z_{t}^2)-2\cov(Z_t,Z_{t+\delta})=2\left(v-\E(Z_tZ_{t+\delta})\right)\;.\nonumber
\eeq

Since $Z$ is bounded, one can rely on Fubini's theorem for all $t>0$ and exchange time-integration and expectation,

\beqn
\E(Z_tZ_{t+\delta})&=&\E\left(Z_t\left(Z_{t}-\int_{t}^{t+\delta}\frac{Z_s}{s}ds+\int_{t}^{t+\delta}\sigma(s,Z_s)dW_s\right)\right)\nonumber\\
&=&\E(Z^2_t)-\E\left(Z_{t}\int_{t}^{t+\delta}\frac{Z_s}{s}ds\right)+\E\left(\int_{t}^{t+\delta}Z_t\sigma(s,Z_s)dW_s\right)\nonumber\\
&=&v-\int_{t}^{t+\delta}\frac{\E\left(Z_{t}Z_s\right)}{s}ds\; \nonumber
\eeqn
(where we have used the fact that $\frac{1}{\sqrt{s}}\sqrt{1-Z_s^2}$ is bounded).

Hence, $\var(Z_{t+\delta}-Z_{t})=2\left(v-f(t,t+\delta)\right)$ where $f(t,s):=\E(Z_tZ_s)$ solves the ODE
\beq
\frac{\partial f(t,s)}{\partial s}=-\frac{f(t,s)}{s}\;.\nonumber
\eeq
Using the initial condition $f(t,t)=v$, the solution is $f(t,s)=vt/s$. Finally, $\lim_{t\to\infty} f(t,t+\delta)=\lim_{t\to\infty} vt/(t+\delta)=v$ showing that $\lim_{t\to 0}\var(Z_{t+\delta}-Z_{t})=0$.

%
%
%

\end{proof}

The activity result can be generalized to the following lemma.

\begin{lemma}
Let $X_t=x_0+\int_0^t \theta_s dW_s$ and suppose $X=(X_t)_{t\geq 0}$ is a bounded non-vanishing martingale in the sense that for all $t\geq 0$, $a\leq X_t\leq b$ and $\mathbb{P}(\theta_t=0)<1$. Then, the path activity of $X$ is collapsing to zero as time passes.

\end{lemma}

\begin{proof}
Since martingales have uncorrelated increments, the variance of increments is the increment of the variances:
\beq
\var(X_{t+\delta}-X_{t})=\var(X_{t+\delta})-\var(X_{t})\nonumber\;.
\eeq
Because the diffusion coefficient $\theta_s$ does not vanish on $(t,t+\delta)$,\\  $\var(X_{t+\delta}-X_{t})=\int_t^{t+\delta}\E(\theta_s^2)ds>0$ showing that the variance of $X_t$ is monotonically increasing with respect to  $t$. But the variance of a bounded process is bounded. In particular, it is easy to see that $\var(X_t)\leq (x_0-a)(b-x_0)$ since $\E(X_t)=x_0$ and the variance of any random variable $Y$ with expectation $\mu_Y$ and taking values in $[a,b]$ is bounded from above by the variance of $a+(b-a)B$ where $B$ is a Bernoulli random variable with parameter $\pi=(\mu_Y-a)/(b-a)$. Hence, $\var(X_t)$ and $\var(X_{t+\delta})$ are increasing to the same limit, proving that for all $\epsilon>0$ there exists $t^\star$ such that $\var(X_{t+\delta}-X_t)<\epsilon$ for all $t>t^\star$.
\end{proof}

We now illustrate the limiting distribution results with a numerical simulation. We simulate the same process as before but conditional on an initial condition at a given time. In particular, we plot in Figure \ref{fig:tran1} the histograms of the transition densities $p_{X_{100y}|X_{90y}}(\cdot;0)$ and  $p_{X_{100y}|X_{90y}}(\cdot;0)$.

\begin{figure}[h!]
\centering
\includegraphics[width=0.45\columnwidth]{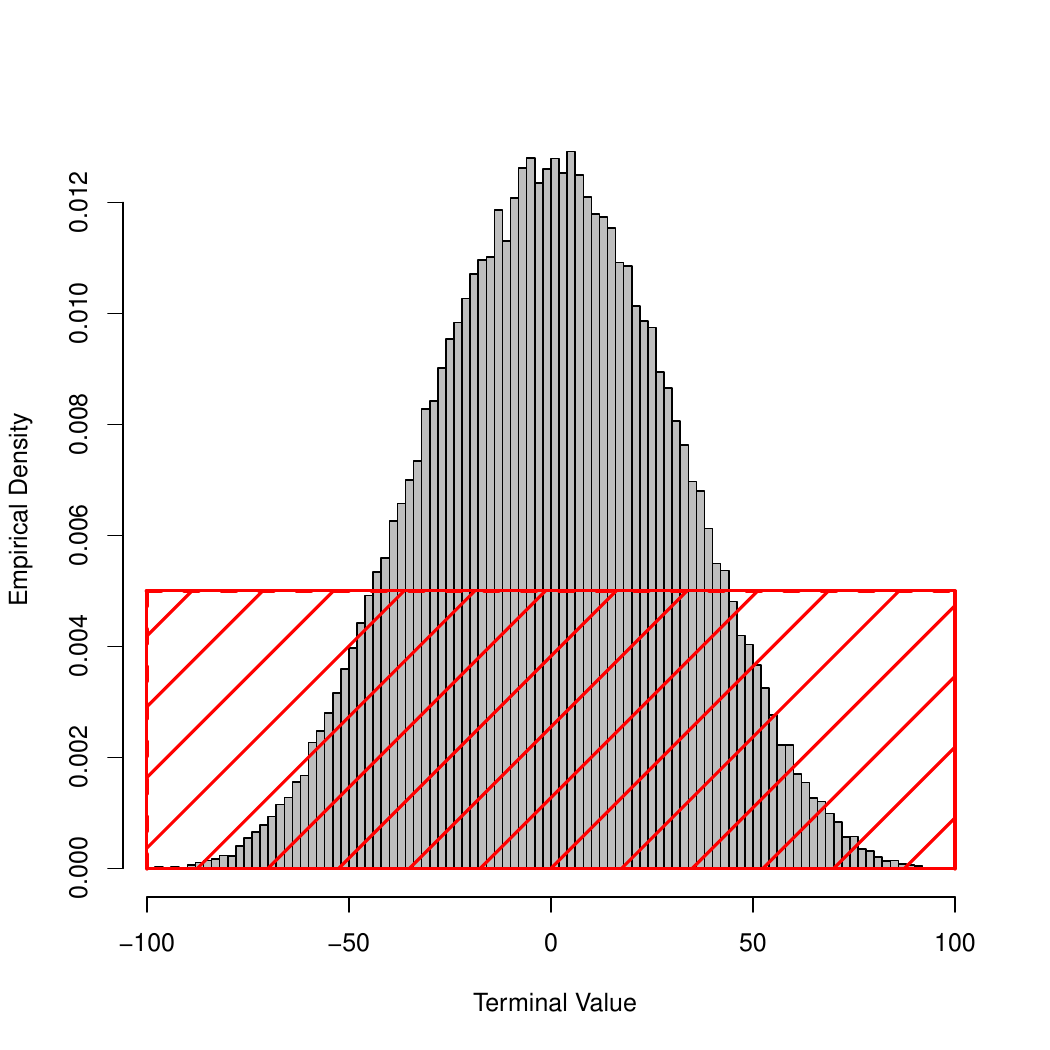}\hspace{0.2cm}
\includegraphics[width=0.45\columnwidth]{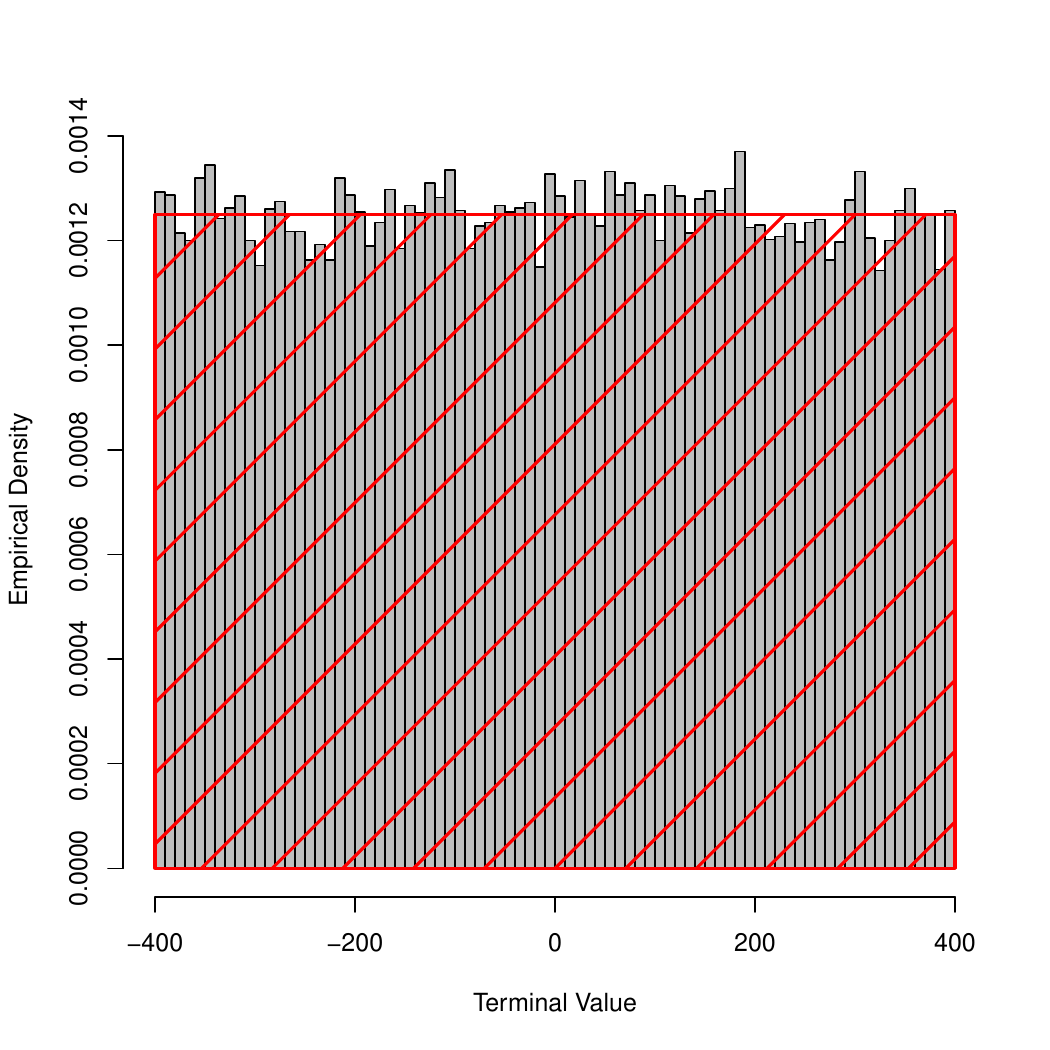}
\caption{Histograms for the transition densities of the solution $X$ of SDE \eqref{eq:martinX2kt} conditional on $X_s=0$ at time $s=90$ years. We take $k=1$ and use the Euler scheme. Left hand side: density $p_{X_{100y}|X_{90y}}(\cdot;0)$ at time $t=100$ years; right hand side: density $p_{X_{400y}|X_{90y}}(\cdot;0)$ at time $t=400y$. }\label{fig:tran1}
\end{figure}

Our simulation describes effectively our earlier results. For $p_{X_{100y}|X_{90y}}(\cdot;0)$ we condition on time in $90$ years, very far away in the future. Given the slowing activity of the SDE solution process, the process will move very slowly after 90 years. Indeed, in the time it takes to get 10 years further  it shows a conditional density for the next ten years, at 100 years, that seems qualitatively Gaussian. This is compatible with the process being so slow as to behave not too differently from an arithmetic Brownian motion qualitatively. Still, our limit-law results tell us that in the very long run the conditional density should go back to uniform. Indeed, this is illustrated in the simulated density $p_{X_{400y}|X_{90y}}(\cdot;0)$. We see that if we wait long enough, 310 years in this case, the density goes back to uniform.

\begin{remark}[Other boundaries]
One could choose time-boundaries that are concave and converge asymptotically to a constant value $B$, e.g. $b(t)=B t /(t+\beta)$ or $b(t)=B(1-e^{-\beta t})$ where $B>0,\beta>0$. It is also possible to use convex boundaries, like e.g. $b(t)=k(e^{\beta t}-1)$, $k>0, \beta>0$. Finally, as mentioned earlier, we could study boundaries of the form $b(t) = k t^\alpha$, $\alpha > 1/2, \ \ k>0$, since in this case too existence and uniqueness of the SDE strong solution is guaranteed.
\end{remark}

\section{Conclusions and further research}\label{sec:conclusions}
We introduced a way to design Stochastic Differential Equations of diffusion type admitting a unique strong solution distributed as a uniform law with conic time-boundaries.  While the result with general boundary is new and conditions for pathwise uniqueness of solutions are new, existence for the cases with square-root and linear boundaries had been dealt with previously in the peacocks literature. We further discussed our results in relation to the peacocks literature. 
We introduced also general mean-reverting diffusion processes having constant uniform margins at all times and showed limit-law theorems establishing that the transition densities also tend to uniform distributions after a long time. In doing so we derived the exact transition densities of the mean-reverting uniform-margins diffusions, and by re-scaling, the exact transition densities of the uniform peacock SDEs we derived initially. Our results may be used to model random probabilities, random recovery rates or random correlations.


\subsection*{Acknowledgements}
The research of Monique Jeanblanc is supported by Chair Markets in Transition  (F\'ed\'eration Bancaire Fran\c caise) and  Labex ANR 11-LABX-0019.
A visit of Frederic Vrins at Imperial College London that contributed to the development of this paper has been funded by the department of Mathematics at Imperial College London with the Research Impulse Grant DRI033DB.

The authors are grateful to C. Profeta for stimulating discussions and for suggesting the proof of Theorem \ref{th:attain}, improving the previous proof of the authors. The authors are grateful to Andreas Eberle for suggestions and correspondence on the proof of the limit law results.

\newpage

\begin{appendix}

\section*{Appendix}

\section{Proof that the solution of the peacock SDE \eqref{eq:martinX2} has uniform law}\label{app:moments}

We start with the following

\begin{definition}
A probability measure $\mu$ is \textit{determined by its moments} when it is the unique probability measure having this set of moments.
\end{definition}

\begin{lemma}
The continuous uniform distribution on $[a,b]$, $-\infty<a<b<\infty$, is determined by its moments.
\end{lemma}
\begin{proof}
Let us note $\alpha_k(p):=\int_{-\infty}^\infty x^k p(x)dx$ the $k$-th moment associated to a probability density function $p$. From Theorem 30.1 of~\cite{billingsley95}, it is known that if all the moments $\alpha_1(p),\alpha_2(p),\ldots $ are finite and are such that the series

$$S_r(p):=\sum_{k=1}^\infty \frac{\alpha_k(p) r^k}{k!}$$

admits a positive radius of convergence, then $p$ is \textit{determined by its moments}.


One concludes from this theorem that if a random variable $X$ satisfies $\E(X^k)=\alpha_k(p)$ for all $k\in\mathbb{N}$, then $X\sim p$ provided that (i) $|\alpha_k(p)|<\infty$ for $k\in\mathbb{N}$ and (ii) there exists $r>0$ such that the series $S_r(p)$ converges.

In particular, if the uniform density in $[a,b]$, $\rho(x):=\frac{1}{b-a}\ind_{\{a\leq x\leq b\}}$, satisfies (i) and (ii), then any random variable $X$ satisfying $\E(X^k)=\alpha_k(\rho)$ for all $k\in\{1,2,\ldots\}$ is uniformly distributed on $[a,b]$.

Let us show that (i) and (ii) are satisfied for the uniform density in $[a,b]$. Condition (i) is clearly met since the moments of the uniform distribution are finite. In particular, defining $c:=|a|\vee|b|$ one has $|\alpha_k(\rho)|\leq c^k<\infty$. On the other hand, for $r>0$,

$$0\leq \frac{|\alpha_k(\rho)|r^k}{k!}\leq \frac{(cr)^k}{k!}\;.$$

Since the series

$$S'_r:=\sum_{k=1}^\infty \frac{(cr)^k}{k!}$$

converges to $e^{cr}-1$,
the series $S_r(\rho)$ converges, too.
This shows that both conditions (i) and (ii) are met for $p=\rho$, and completes the proof.
\end{proof}

\begin{theorem}\label{th:momentsuni}
The solution $X$ to the SDE~(\ref{eq:martinX2}) is a uniform martingale on $[-b(t),b(t)]$ in the sense that for all $t>0$, $X_t$ is uniformly distributed on $[-b(t),b(t)]$.
\end{theorem}
\begin{proof}
From the above results, it is enough to show that all moments of the random variable $X_t$ ($t>0$) associated to eq. (\ref{eq:martinX2}) coincide with those of the density $\ind_{\{-b(t)\leq x\leq b(t)\}}\frac{1}{2b(t)}$.

Let $X$ be a random variable uniformly distributed on $[a,b]$. Then,

$$\E(X^n)=\frac{1}{n+1}\sum_{i=0}^n (a)^i(b)^{n-i}\;.$$

In the special case where $a=-b$, this expression reduces to

	\beqn
	\E(X^n) = \left\{\begin{array}{ll}
			\frac{b^n}{n+1} 	& \hbox{ if $n$ is odd } \\
			0		& \hbox{ otherwise. }
				\end{array}
		  \right.\nonumber
	\eeqn

Let us now compute the moments of

$$X_t=\int_0^t\sigma(X_s,s)dW_s~,\qquad\sigma(t,x)=\ind_{\{ -b(t)\leq x\leq b(t)\}} \sqrt{\frac{\dot{b}(t)}{b(t)}}\sqrt{b^2(t)-x^2}$$

solving eq.~(\ref{eq:martinX2}). By It\^o's lemma:

$$X_t^n=n\int_0^t X_s^{n-1}dX_s+\frac{1}{2}n(n-1)\int_0^t X_s^{n-2}\sigma^2(X_s,s)ds$$

and we can compute the expression for the $n$-th moment, $n\geq 2$ using a recursion. Using the property that It\^o 's  integrals have zero expectation and exchanging integration and expectation operators, which is possible since  $X_s^{n-2}\sigma^2(X_s,s)$ is bounded for all $s$ and $n\geq 2$, we obtain

\beqn
\E(X_t^n)&=&n\E\left(\int_0^t X_s^{n-1}dX_s\right)+\frac{1}{2}n(n-1)\E\left(\int_0^t X_s^{n-2}\sigma^2(X_s,s)ds\right)\nonumber\\
&=&\frac{n(n-1)}{2}\left(\int_0^t b(s)\dot{b}(s)\E(X_s^{n-2})ds-\int_0^t \frac{\dot{b}(s)}{b(s)}\E(X_s^{n})ds\right)\label{eq:unifmoment}
\eeqn

Notice that we have postulated in the last equality that the indicator $\ind_{\{ -b(s)\leq X_s\leq b(s)\}} $ in $\sigma(t,x)$ is always 1. This is a natural assumption: it says that $X$ cannot stay on a boundary with a strict positive probability for a given period of time. This happens because in case $X$ reaches $\pm b(t)$ at some time $t$, the process is locally frozen ($\sigma(t,x)=0$) but the boundary $b(t)$ keeps on growing.

Obviously, $\E(X_t)=X_0=0$ since $X$ is a martingale and one concludes from eq.~(\ref{eq:unifmoment}) that the $n$-th moment of $X_t$ is zero when $n$ odd. For $n$ even, eq.~(\ref{eq:unifmoment}) can be written as
\beq
M_n(t)=\frac{n(n-1)}{2}\left(\int_0^t b(s)\dot{b}(s) M_{n-2}(s)ds-\int_0^t \frac{\dot{b}(s)}{b(s)}M_n(s)ds\right)\nonumber
\eeq

with $M_n(t):=\E(X_t^n)$. This can be written as a recursive differential equation
\beq
\frac{\partial M_n(t)}{\partial t}=\dot{b}(t)\frac{n(n-1)}{2}\left(b(t) M_{n-2}(t)- \frac{1}{b(t)}M_n(t)\right)\nonumber
\eeq

with the constraint that $f(t,0)=\E(X_t^0)=1$. The solution to this equation is $M_n(t)=b^n(t)/(n+1)$. One concludes that $X_t$ is uniform on $[-b(t),b(t)]$ since all the odd  moments are zero and all the even moments are given by
$$\E(X_t^n)=\frac{b^n(t)}{n+1}$$
and agree with those of a random variable uniformly distributed on $[-b(t),b(t)]$.
\end{proof}

\section{First conditional moments}

The first six conditional moments are

\begin{eqnarray}
M_1(s,t;z)&=&z\frac{b(s)}{b(t)}\nonumber\\
M_2(s,t;z)&=&\frac{1}{3}+\left(z^2-\frac{1}{3}\right)\left(\frac{b(s)}{b(t)}\right)^3\nonumber\\
M_3(s,t;z)&=&\frac{3}{5}z\left(\frac{b(s)}{b(t)}-\left(\frac{b(s)}{b(t)}\right)^6\right)+z^3\left(\frac{b(s)}{b(t)}\right)^6\nonumber\\
M_4(s,t;z)&=&\frac{1}{5}+\left(z^4-\frac{1}{5}\right)\left(\frac{b(s)}{b(t)}\right)^{10}+\frac{6}{7}\left(z^2-\frac{1}{3}\right)\left(\left(\frac{b(s)}{b(t)}\right)^3-\left(\frac{b(s)}{b(t)}\right)^{10}\right)\nonumber\\
M_5(s,t;z)&=&\frac{1}{21}z\left(9\frac{b(s)}{b(t)}-14\left(\frac{b(s)}{b(t)}\right)^6+5\left(\frac{b(s)}{b(t)}\right)^{15}\right)\nonumber\\
&&+\frac{10}{9}z^3\left(\left(\frac{b(s)}{b(t)}\right)^6-\left(\frac{b(s)}{b(t)}\right)^{15}\right)+z^5\left(\frac{b(s)}{b(t)}\right)^{15}\nonumber\\
M_6(s,t;z)&=&\frac{1}{7}+\left(z^6-\frac{1}{7}\right)\left(\frac{b(s)}{b(t)}\right)^{21}
+ \frac{15}{11}\left(z^4-\frac{1}{5}\right)\left(\left(\frac{b(s)}{b(t)}\right)^{10}-\left(\frac{b(s)}{b(t)}\right)^{21}\right)\nonumber\\
& & +\frac{5}{77}\left(z^2-\frac{1}{3}\right)\left(11 \left(\frac{b(s)}{b(t)}\right)^3-18 \left(\frac{b(s)}{b(t)}\right)^{10} + 7 \left(\frac{b(s)}{b(t)}\right)^{21}\right)\nonumber
\end{eqnarray}

The first six $\alpha$ matrices are

\[\alpha[1]=\alpha[2]=
\left[\begin{array}{ll}
1
\end{array}\right]\;,\;\alpha[3]=
\left[\begin{array}{ll}
3/5&0\\
3/5&1
\end{array}\right]\;,\;\alpha[4]=
\left[\begin{array}{ll}
6/7&0\\
6/7&1
\end{array}\right]\;,\;
\]
\[\alpha[5]=
\left[\begin{array}{lll}
3/7&0&0\\
2/3&10/9&0\\
5/21&10/9&1
\end{array}\right]\;,\;\alpha[6]=
\left[\begin{array}{lll}
5/7&0&0\\
90/77&15/11&0\\
35/77&15/11&1
\end{array}\right]\;.
\]

\section{Other uniform diffusions}\label{app:sampling}

Let  $(F(\cdot;t);t \geq 0)$ be a set of time-indexed invertible CDFs with densities $f(y;t)=\frac{\partial F(y;t)}{\partial y}$ and $G(\cdot;t)$ the inverse of $F(\cdot;t)$ satisfying $G(F(x;t);t)=x$ for all $x$ and all $t\geq0$. The stochastic process $U_t:=(1+\tilde{Z}_t)/2$ (where $\tilde{Z}$ is the solution of \eqref{eq:unifZ}) is uniform in $[0,1]$. Setting $Y_t:=G(U_t;t)$ (so that $F(Y_t;t)= U_t$), the stochastic process $Y$ has time-$t$ marginal CDFs $F(\cdot;t)$ and its dynamics are given by (to check)
\begin{eqnarray}
dY_t&=&\frac{\partial G(U_t;t)}{\partial t}dt+\frac{1}{f\left(G(U_t;t);t\right)}dU_t-\frac{1}{2}\frac{1}{f^2\left(G(U_t;t);t\right)}\frac{1}{f\left(G(U_t;t);t\right)}d\langle U\rangle_t\nonumber\\
&=&G_t(F(Y_t;t)_t;t)dt+\frac{1}{2f\left(Y_t;t\right)}d\tilde{Z}_t-\frac{1}{8}\frac{1}{f^3\left(Y_t;t\right)}d\langle \tilde{Z}\rangle_t\nonumber\\
&=&\left(G_t(F(Y_t;t);t)+\frac{\dot{b}(t)}{2b(t)f(Y_t;t)}(1-2F(Y_t;t))+\frac{\dot{b}(t)}{b(t)f^3(Y_t;t)}F(Y_t;t)(F(Y_t;t)-1)\right)dt\nonumber\\
&&\hspace{6cm}+\sqrt{\frac{2\dot{b}(t)}{b(t)f^2(Y_t;t)}F(Y_t;t)(1-F(Y_t;t))}dW_t
\end{eqnarray}

What is striking is that the martingale that is uniform in the expanding boundary $ t \mapsto [-b(t),b(t)]$ seems to be essentially unique, in the sense that there is only one diffusion coefficient that will make the diffusion martingale attain a  uniform law in $[-b(t),b(t)]$. One can check this informally by inspecting the ``invert the Fokker-Planck-Kolmogorov'' equation approach we adopted. However, there would be many diffusions with uniform margins in general. Indeed, there are for example many $b(t)$ that would lead to a uniform $Z$ in $(-1,1)$. More generally, we can find uniform diffusions whose drift and diffusion coefficients take a completely different form with respect to the ``$\dot{b}/b$'' proportional drift of Section \ref{sec:meanrev}. We now give an example.

\begin{proposition}
Let $W$ be a standard Brownian motion. Define  $Z_t:=2\Phi\left(\frac{W_t}{\sqrt{t}}\right)-1$. $Z$ is a stochastic process with uniform distribution in $(-1,1)$ at all times (possibly with random initialization $Z_{t_0}=\zeta$ at $t_0>0$). The dynamics of $Z$ are given by

\begin{eqnarray}
dZ_t&=&\underbrace{-\frac{2}{t}\Phi^{-1}\left(\frac{1+Z_t}{2}\right)\varphi\left(\Phi^{-1}\left(\frac{1+Z_t}{2}\right)\right)}_{\mu(t,x)}dt+\underbrace{\frac{2}{\sqrt{t}}\varphi\left(\Phi^{-1}\left(\frac{1+Z_t}{2}\right)\right)}_{\sigma(t,x)}dW_t\nonumber   .
\end{eqnarray}

It can be shown that this satisfies the Forward-Kolmogorov equation with

$p(x,t)= (1/2) \ind_{x\in(-1,1)}$ as
$$\mu_x(t,x)=(\sigma_x(t,x))^2+\sigma(t,x)\sigma_{xx}(t,x)=\left(\left(\Phi^{-1}\left(\frac{1+Z_t}{2}\right)\right)^2-1\right)/t$$

Moreover, the law of $Z_t|Z_s$ tends to that of a Uniform in $(-1,1)$ as $t\to\infty$.
\end{proposition}

\begin{proof}
Conditioning upon $Z_s$ is equivalent to conditioning w.r.t. $W_s$ as $Z_s=2\Phi\left(\frac{W_s}{\sqrt{s}}\right)-1$. But $\left.\frac{W_t}{\sqrt{t}}\right|_{W_s}\sim\frac{W_s+\sqrt{t-s}Z}{\sqrt{t}}$ where $Z\sim\mathcal{N}(0,1)$. Hence,
$$\Pb(Z_t\leq x|Z_s)=\Pb\left(\left.\Phi\left(\frac{W_t}{\sqrt{t}}\right)\leq \frac{1+x}{2}\right|W_s\right)=\Pb\left(\left.\Phi\left(\frac{W_s+\sqrt{t-s}Z}{\sqrt{t}}\right)\leq \frac{1+x}{2}\right|W_s\right)$$

Inverting the standard Normal CDF,

$$\Pb(Z_t\leq x|Z_s)=\Pb\left(Z\leq \frac{\sqrt{t}\Phi^{-1}\left(\frac{1+x}{2}\right)-W_s}{\sqrt{t-s}}\right)=\Phi\left(\frac{\sqrt{t}\Phi^{-1}\left(\frac{1+x}{2}\right)-W_s}{\sqrt{t-s}}\right)$$

So $\Pb(Z_t\leq x|Z_s)\to (1+x)/2$ as $t\to\infty$ for all $(s,Z_s)$ where $Z_s\in(-1,1)$.
\end{proof}

\end{appendix}

\end{document}